\documentclass[11pt]{article}

 \usepackage{amscd,amsmath,amsfonts,amssymb,amsthm,cite,mathrsfs,color} 
\usepackage{dsfont} \usepackage{leftidx}
\usepackage{tikz} \usepackage{enumerate} 
 \usepackage{appendix}

 \usepackage[left=2.5cm, right=2.5cm, top=2.5cm, bottom=2.5cm]{geometry}
\usepackage{graphicx} \usepackage{xcolor}
\usepackage[colorlinks,linkcolor=red]{hyperref}
\numberwithin{equation}{section} 

\usepackage{amsthm}

\theoremstyle{plain}
\newtheorem{theorem}{Theorem}[section]
\newtheorem{lemma}[theorem]{Lemma}

\newtheorem{proposition}[theorem]{Proposition}

\theoremstyle{definition}
\newtheorem{definition}[theorem]{Definition}

\newtheorem{remark}[theorem]{Remark}

\begin{document}
\title{\bf{ 
Bulk universality for deformed GinUEs
}}

\author{Lu Zhang\footnotemark[1]}
\renewcommand{\thefootnote}{\fnsymbol{footnote}}
\footnotetext[1]{School of Mathematical Sciences, University of Science and Technology of China, Hefei 230026, P.R.~China. E-mail: zl123456@mail.ustc.edu.cn}


 \maketitle
 \begin{abstract}
For the deformed complex Ginibre ensemble with a mean  normal matrix, under certain assumptions on the mean matrix we prove that the same bulk statistics holds as in the complex Ginibre matrix bulk. This is  the continuation of the previous joint papers \cite{LZ23} and \cite{LZ24}, which deal with local eigenvalue statistics at the edge.

 \end{abstract}

\section{Introduction and main results}

\subsection{Introduction}
The deformed complex Ginibre ensembles are defined as follows.
 \begin{definition}  \label{GinU} 
A random  complex  
$N\times N$ matrix  $X$, is said to belong to the deformed complex  Ginibre ensemble with mean matrix $X_0={\rm diag}\left( A_0,0_{N-r} \right)$ and  time parameter $\tau>0$, denoted by GinUE$_{N}(A_0)$,  if  the joint probability density function for  matrix entries  is given by   
\begin{equation}\label{model}
P_{N}(A_0;X)=    \Big(\frac{N}{\pi\tau}\Big)^{N^2}\
e^{ -\frac{N}{\tau} {\rm Tr} (X-X_0)(X-X_0)^*}.
\end{equation}
 Here $A_0$ is a $r\times r$ constant  matrix. 

\end{definition}
The study of the deformed model \eqref{model}  is motivated mostly
from the effort toward the understanding of the effect of perturbations on the spectrum of a large-dimensional random matrix, see e.g. \cite{BC16}. 
In the non-perturbed case,  the study of  non-Hermitian random matrices  was first initiated  by  Ginibre  \cite{Gi}  for those  matrices with i.i.d. real/complex standard  Gaussian entries, and then was extended to  i.i.d.  entries. At a macroscopic  level, the limit spectra  measure was the famous \textit{circular law}, which  is   a uniform measure on the unit disk, see for e.g. a recent survey \cite{BC12} and references therein for more details. At a microscopic level,  local  eigenvalue statistics  in the bulk and near the edge  of the spectrum has already many tremendous results, see e.g. \cite{TV15}, \cite{CES} and very recently \cite{MO}.

Although the deformed GUEs  or the general deformed Hermitian models   have  been  
quite well understood (see e.g. \cite{CP16} and references therein),  not much is known about   the  local  eigenvalue statistics  in the non-Hermitian situation. Also, for the deformed model  \eqref{GinU},    the  circular law still holds  under a  small low rank perturbation from $X_0$, see   \cite[Corollary 1.12]{TV10} as the culmination of work by many authors and a survey \cite{BC12} for further details. For general $X_0$ with infinite rank but certain limiting behaviour,  the limit spectra  measure no longer follows the circular law, while the boundary curve of its support can still satisfy good parametric equation, see e.g. \eqref{supp mu} or \cite[Proposition 1.2]{BC16}. These curves can exhibit different geometries from the circle, like degeneracy, connectedness and etc.
 
Our goal is to   identify all the possible local  eigenvalue statistics at the edge and in the bulk  of the spectrum, for the deformed model  \eqref{GinU} with mean normal matrix $X_0$ under certain restriction \eqref{A0 form}  on $A_0$ which has infinite rank. The series are divided in three papers. 
This is the third paper, which investigates the local eigenvalue statistics near the bulk point of the large $N$ limit spectrum distribution of $X$. In contrast, the first two papers \cite{LZ23} and \cite{LZ24} investigates the local eigenvalue statistics near the edge point, which is further classified to be whether critical or not, according to its degeneracy at the boundary curve.

It is a classical result that for complex Ginibre unitary ensembles, i.e. model \eqref{GinU} with zero mean, local eigenvalue statistics near the bulk point forms a class of determinantal point processes, with correlation kernels $K(z,w)=\frac{1}{\pi} e^{-(|z|^2+|w|^2)/2+z\overline{w}}$, see e.g. \cite{Me}. In \cite{TV15} Tao and Vu generalise this result to more general complex random matrix, which matches moments with complex Ginibre unitary ensembles to fourth order and satisfies other regular conditions. Very recently in \cite{MO}, Maltsev and Osman proved the Ginibre bulk statistics both for model \eqref{GinU} with certain restrictions on $X_0$ in the local law setting and a rather general ensembles of random complex matrix with i.i.d. entries with finite moments  such that the real and imaginary parts are independent. Our present paper will show that under certain restriction \eqref{A0 form} on the mean matrix $X_0$ with infinite rank, the large $N$ limit local eigenvalue statistics for model \eqref{GinU} is still governed by the Ginibre bulk statistics.

\subsection{Main results}

Throughout the present  paper, we select $A_0$ as a normal matrix with the following form
\begin{small}
\begin{equation}\label{A0 form}
A_0={\rm diag}\left(a_1\mathbb{I}_{r_1},\cdots,a_t\mathbb{I}_{r_t},
z_0\mathbb{I}_{r_0},A_{t+1}
\right),
\end{equation}
\end{small} 
where $t$ is a given non-negative integer,  $z_0$ is  a spectral parameter and   $A_{t+1}$ is a normal $r_{t+1}\times r_{t+1}$ matrix with $z_0$  not an  eigenvalue.
 Put $R_0=N-r$ where  $r=\sum_{\alpha=0}^{t+1} r_{\alpha}$, we assume that    
 $r_0$, $r_{t+1}$ and $R_0$ are  finite, independent of $N$, and  as $N\to \infty$
\begin{small}
 \begin{equation}\label{ralpha N}
r_{\alpha}=c_{\alpha}N+R_{\alpha,N},
\quad R_{\alpha,N}=O(1), \quad \alpha=1,\cdots,t.
\end{equation}
\end{small}
So we have 
\begin{small}
\begin{equation}\label{calpha sum}
\sum_{\alpha=1}^t  c_{\alpha}=1.
\end{equation}
\end{small}

For given complex numbers $a_1,\cdots,a_t$,  introduce  a  probabililty measure  on the complex plane 
\begin{small}
\begin{equation*}
d\nu(z) =    \sum_{\alpha=1}^t c_{\alpha} \delta(z-a_{\alpha}),
\end{equation*}
\end{small}
and a relevant expectation
\begin{small}
\begin{equation}\label{parameter}
P_{00}(z_0):= \int  \frac{1}{|z-z_0|^2} d\nu(z), 
 \end{equation}
\end{small}
the support of limit spectral measure $\mu_{\infty}$ is
\begin{small}
\begin{equation}\label{supp mu} 
\mathrm{Supp}(\mu_{\infty}):=\Big\{z_0\in \mathbb{C}: P_{00}(z_0)\geq \frac{1}{\tau} \Big\};
\end{equation}
\end{small}
see e.g.  \cite[Proposition 1.2]{BC16}. In this paper we consider the bulk point $z_0$, that is $P_{00}(z_0)> \frac{1}{\tau}$. Then there exists a unique positive $t_0$, such that
\begin{small}
\begin{equation}\label{bulkparameter}
\sum_{\alpha=1}^t \frac{\tau c_{\alpha}}
{|a_{\alpha}-z_0|^2+t_0}=1.
\end{equation}
\end{small}
Further introduce two relevant expectations
\begin{small}
\begin{equation}\label{parameter2}
  P_0:=P_{0}(z_0)=  
  \int  \frac{z-z_0}{(|z-z_0|^2+t_0)^2} d\nu(z), \quad 
   P_1:=P_{1}(z_0)=  
  \int  \frac{1}{(|z-z_0|^2+t_0)^2} d\nu(z)
\end{equation}
\end{small}
These  quantities  are important  to characterize the rescaling of the spectral points which are parameters in the n-point correlation functions.

Now, we are ready to state our main results.
\begin{theorem}\label{2-complex-correlation}
For the $ {\mathrm{GinUE}}_{N}(A_0)$ ensemble with $R_0\geq n$ and the assumptions  
\eqref{A0 form} on $A_0$, if $z_0$ is a bulk point which satisfies $a_{\alpha}\not=z_0$ for $\alpha=1,\cdots,t$ and $\sum_{\alpha=1}^t
\frac{\tau c_{\alpha}}{|z_0-a_{\alpha}|^2+t_0}=1$ with $t_0>0$, then 
as $N\to \infty$ scaled eigenvalue correlations  hold uniformly for   all 
$\hat{z}_{1}, \ldots, \hat{z}_{n} $
in a compact subset of $\mathbb{C}$. 
\begin{small}
\begin{multline} \label{corre2edgecomplex noncritical}
\frac{1}{N^n\Big(  
t_0P_1+\frac{|P_0|^2}{P_1}
 \Big)^n}
R_N^{(n)}\bigg(
A_0;z_0+\Big( N\Big(  
t_0P_1+\frac{|P_0|^2}{P_1}
 \Big) \Big)^{-\frac{1}{2}}\hat{z}_1
,\cdots,
z_0+\Big( N\Big(  
t_0P_1+\frac{|P_0|^2}{P_1}
 \Big) \Big)^{-\frac{1}{2}}\hat{z}_n
\bigg)\\
=
\det\left(
\Big[
\frac{1}{\pi}e^{-\frac{1}{2}
|\hat{z}_i|^2-\frac{1}{2}
|\hat{z}_j|^2+\hat{z}_i\overline{\hat{z}_j}
}
\Big]_{i,j=1}^n
\right)
+O\big( N^{-\frac{1}{2}} \big).
\end{multline}
\end{small}

\end{theorem}

\begin{remark}
Note in the right hand side of \eqref{corre2edgecomplex noncritical}, the time parameter $\tau$ does not exist, while it appears in the rescaling 
coefficients \begin{small}
$ \Big(  
t_0P_1+\frac{|P_0|^2}{P_1}
 \Big)^{-\frac{1}{2}} $
\end{small}, which is indicated by \eqref{bulkparameter} and \eqref{parameter2}. By setting $n=1$ in \eqref{corre2edgecomplex noncritical}, we get the asymptotic formula for the mean level density
\begin{equation}\label{meanleveldensity}
\frac{1}{N}
R_N^{(1)}\big(
A_0;z_0,\cdots,z_0\big)=\frac{t_0P_1}{\pi}+\frac{|P_0|^2}{\pi P_1}
+O\big( N^{-\frac{1}{2}} \big),
\end{equation} 
this result coincides with \cite[Equation (1.12)]{BPD}, which describes the density of probability measure of the sum of the $\star$-free circular and normal operators, in the terminology of free probability theory.
\end{remark}

\begin{remark}
 The condition $ R_0\geq n $ comes from \cite[Proposition 1.3]{LZ22}, which provides an integral representation for the n-point correlations of GinUE$_{N}(A_0)$ in Definition \ref{GinU}, the appearance of this condition comes from the fact when proving \cite[Proposition 1.3]{LZ22}, we use induction on the matrix dimension $N$, and needs to reduce the n-point correlation function to a matrix integral consisting of expectation with respect to the same deformed Ginibre ensemble but with dimension $N-n$; This means the mean matrix $X_0={\rm diag}\left( A_0,0_{N-r=R_0} \right)$(cf. Definition \ref{GinU}) should have enough zeros to fit the truncation. The condition $a_{\alpha}\not=z_0$ for $\alpha=1,\cdots,t$ is essential, because by \eqref{A0 form}
we know the testing point $z_0$ is an eigenvalue of the mean matrix $X_0$,  if the condition holds, then the geometric multiplicity of $z_0$ is finite independent of $N$, otherwise it will tend to infinity as the dimension of the random matrix tends to infinity. Technically, the condition $a_{\alpha}\not=z_0$ for $\alpha=1,\cdots,t$ is necessary in the proof of Proposition \ref{RNn delta}, to be more precise, in deriving the step \eqref{error term2}.
\end{remark}

\begin{remark}
As having discussed with the authors of \cite{MO}, we must admit there is overlap between our main result Theorem \ref{2-complex-correlation} and \cite[Theorem 1.1]{MO}. While we believe this paper still has its own value, as the methodologies between our paper and \cite{MO} are different. In the proof of \cite[Theorem 1.1]{MO}, the authors apply the partial Schur decomposition to express the n-point correlation function as  a multiple integral which involves matrix and spherical integrals and an expectation with respect to the product of characteristic polynomials, see \cite[Equation (3.19)]{MO}; Although our idea in this paper is similar, we apply a complete(rather than partial) Schur decomposition to express the n-point correlation function as  a multiple integral which involves matrix integrals and an expectation with respect to the product of characteristic polynomials, this has been done in \cite[Equation (1.9)]{LZ22}(which includes me as one of the authors), we claim that \cite[Equation (3.19)]{MO} and \cite[Equation (1.9)]{LZ22} are different, also \cite[Equation (1.9)]{LZ22} is an earlier result. To proceed the analysis, the authors of \cite{MO} need to apply inequalities of concentration of spherical integral, see \cite[Lemma 6.2\& 6.3]{MO}. In contrast, we are devoted to the Laplace method for matrix integrals, see Proposition \ref{foranlysis}, using of two maximum lemmas in Section \ref{Sectmax} gives us the concentration reduction Proposition \ref{RNn delta}, which restricts the matrix variables in a small region around the maximum point. Then we can do Taylor expansions of functions of the matrix variables. In conclusion, {\bf proofs between our paper and \cite[Theorem 1.1]{MO} are established independently, and are different in many aspects}. Our methods may be more elementary, which could be treated as a classical example of the Laplace method for matrix integrals. The methodology in the proof of Theorem \ref{2-complex-correlation} still applies when considering the edge statistics, see \cite{LZ22,LZ23,LZ24}.
\end{remark}

\section{Integral representation of correlation functions}

\subsection{Notation} \label{sectnotation}
  Denote  the conjugate, transpose,  and  conjugate transpose of a complex matrix $A$ by  $\overline{A}$,  $A^t$ and $A^*$,  respectively.   Besides, the following symbols may be used in this and subsequent sections.  
Denote the tensor product of an $m\times n$  matrix  $A=[a_{i,j}]$  and  a $p\times q$ matrix  $B$ as a block matrix 
\begin{small}
$$
A \otimes B=\left[\begin{smallmatrix}
a_{11}B   &\cdots  &  a_{1n} B   \\ \vdots  & \ddots &   \vdots   \\
a_{m1}B   &\cdots  &  a_{mn} B   
\end{smallmatrix}\right],$$
\end{small}
and denote the Hilbert-Schmidt norm of a complex matrix $M$ as $\|M\|:=\sqrt{{\rm{Tr}}(MM^*)}$.     $M=O(A_N)$ for some matrix sequence $A_N$  means that each  element  has the same order as that of $ A_N$.

\subsection{Matrix variables} \label{sectMatrixvariables}
To prove Theorem \ref{2-complex-correlation}, we need a matrix integral representation of correlation functions, see \eqref{foranlysisequ}. To this end, recalling \eqref{A0 form}, we need to introduce several matrix variables $Y$, $Q_0$, $Q_{t+1}$ and upper triangular $T_{\alpha}$ with non-negative diagonal elements for $\alpha=1,\cdots,t$, which are of sizes $n\times n$, $n\times r_0$, $n\times r_{t+1}$ and $n\times n$ respectively. Introduce scaled spectral points which play the role as parameter in correlation function
\begin{small}
\begin{equation}\label{zi}
Z=z_0 I_{n}+N^{-\frac{1}{2}}\hat{Z}, \quad Z=\mathrm{diag}(z_1, \ldots,z_n), \  \hat{Z}=\mathrm{diag}(\hat{z}_1, \ldots,\hat{z}_n).
\end{equation} 
\end{small}
Denote the following block matrices: for $\alpha=1,\cdots,t$, set
\begin{small}
\begin{equation}\label{A alpha}
A_{\alpha}=
\begin{bmatrix}
\sqrt{\gamma_N}(Z-a_{\alpha}\mathbb{I}_n)  &
-Y^* \\
Y & \sqrt{\gamma_N}
\big(Z^*-\overline{a}_{\alpha}\mathbb{I}_n
\big)
\end{bmatrix},
\end{equation}
\end{small}
\begin{small}
\begin{equation}\label{L1hat}
\widehat{L}_1={\rm diag}\big(
A_1\otimes \mathbb{I}_n,\cdots,
A_t\otimes \mathbb{I}_n,
\widehat{B}_{t+1}
\big),
\end{equation}
\end{small}
\begin{small}
\begin{equation}\label{B0hat}
\widehat{B}_{t+1}=
\begin{bmatrix}
\sqrt{\gamma_N}
\big(
Z\otimes \mathbb{I}_{\widetilde{r}_{t+1}}
-\mathbb{I}_n\otimes \widetilde{A}_{t+1}
\big)  &
-Y^*\otimes \mathbb{I}_{\widetilde{r}_{t+1}} \\
 Y\otimes \mathbb{I}_{\widetilde{r}_{t+1}} &  \sqrt{\gamma_N}
\big(
Z^*\otimes \mathbb{I}_{\widetilde{r}_{t+1}}
-\mathbb{I}_n\otimes \widetilde{A}_{t+1}^*
\big)
\end{bmatrix},
\end{equation}
\end{small}
\begin{small}
\begin{equation}\label{L0hat}
\widehat{L}_2=
\left[
\begin{smallmatrix}
\left[
\left[\begin{smallmatrix}
a_{\alpha}\mathbb{I}_n & \\ & \overline{a}_{\alpha}\mathbb{I}_n
\end{smallmatrix}\right]
\otimes
(T_{\alpha}^*T_{\beta}) \right]_{\alpha,\beta=1}^t
&
\left[
\left[\begin{smallmatrix}
a_{\alpha}\mathbb{I}_n & \\ & \overline{a}_{\alpha}\mathbb{I}_n
\end{smallmatrix}\right]
\otimes
(T_{\alpha}^*\widetilde{Q}_{t+1}) \right]_{\alpha=1}^t
   \\
\left[
\begin{smallmatrix}
\mathbb{I}_n \otimes 
(\widetilde{A}_{t+1}\widetilde{Q}_{t+1}^*
T_{\beta})  &  \\
 &  
\mathbb{I}_n  \otimes 
 (\widetilde{A}_{t+1}^*\widetilde{Q}_{t+1}^*
T_{\beta})
\end{smallmatrix}\right]_{\beta=1}^t
&
\left[\begin{smallmatrix}
\mathbb{I}_n  \otimes 
(\widetilde{A}_{t+1}\widetilde{Q}_{t+1}^*\widetilde{Q}_{t+1})
 &  \\
& 
\mathbb{I}_n  \otimes 
(\widetilde{A}_{t+1}^*\widetilde{Q}_{t+1}^*\widetilde{Q}_{t+1})
\end{smallmatrix}\right]
\end{smallmatrix}\right],
\end{equation}
\end{small}
where
\begin{equation}\label{Q convenience} 
\widetilde{Q}_{t+1}=\left[ Q_0,Q_{t+1}
\right],\quad
\widetilde{A}_{t+1}={\rm diag}\left( z_0\mathbb{I}_{r_0},A_{t+1}
\right),\quad
\widetilde{r}_{t+1}=r_0+r_{t+1},
\end{equation}
whenever $r_0>0$.


 On the other hand, rewrite $T_{\alpha}$ as  sum of  a diagonal matrix $\sqrt{T_{{\rm d},\alpha}}$ and a strictly upper triangular  matrix ${T_{{\rm u},\alpha}}$
\begin{small}
\begin{equation}\label{Talpha 2}
T_{\alpha}=\sqrt{T_{{\rm d},\alpha}}+T_{{\rm u},\alpha}:=\begin{bmatrix}
\sqrt{t_{1,1}^{(\alpha)}} & \cdots & t_{i,j}^{(\alpha)} \\
 & \ddots & \vdots \\
 & & \sqrt{t_{n,n}^{(\alpha)}}
\end{bmatrix},
\end{equation} 
\end{small}
where
\begin{small}
\begin{equation*}
T_{{\rm d},\alpha}={\rm diag}
\big(
t_{1,1}^{(\alpha)},\cdots,t_{n,n}^{(\alpha)}
\big).
\end{equation*}
\end{small}

\subsection{Sketch of the proof for Theorem \ref{2-complex-correlation}} \label{sectsketch}
Combining integral representations for the n-point correlations given in our previous paper \cite[Proposition 1.3]{LZ22} and duality formulas for auto-correlation functions of characteristic polynomials in \cite{Gr, LZ22}, we can write the n-point correlation function $R_N^{(n)}(A_0;z_1,\cdots,z_n)$ in an appropriate form which is ready for asymptotic analysis by the Laplace method argument. This has already been done in \cite[Proposition 2.2]{LZ23}: 
\begin{proposition}\label{foranlysis} 
Let $R_0\geq n$, then
\begin{small}
\begin{equation}\label{foranlysisequ} 
\begin{aligned}
R_N^{(n)}(A_0;z_1,\cdots,z_n)=
\widetilde{C}_{N,n}
\int \cdots \int g(T,Y,Q)\exp\{ Nf(T,Y,Q) \}
{\rm d}Y
{\rm d}Q_0{\rm d}Q_{t+1} \prod_{\alpha=1}^t
{\rm d}T_{\alpha},
\end{aligned}
\end{equation}
\end{small}
where 
\begin{small}
\begin{equation}\label{fTY}
\begin{aligned}
&f(T,Y,Q)=\sum_{\alpha=1}^t
c_{\alpha}\log\det(T_{\alpha}T_{\alpha}^*)
+\frac{1}{\tau}h(T,Q)
-
\frac{N-n}{\tau N}
{\rm Tr}\big(
YY^*
\big)+\sum_{\alpha=1}^tc_{\alpha}\log\det(A_{\alpha}),
\end{aligned}
\end{equation}
\end{small}
with
\begin{small}
\begin{equation}\label{hQrewrite}
\begin{aligned}
&h(T,Q)=\sum_{\alpha=1}^t\left(-|z_0-a_{\alpha}|^2{\rm Tr}\big(
T_{\alpha}T_{\alpha}^*
\big)
+N^{-\frac{1}{2}}
\left(
\overline{a}_{\alpha}{\rm Tr}\big(
\hat{Z} T_{\alpha}T_{\alpha}^*
\big)+a_{\alpha}{\rm Tr}\big(
\hat{Z}^* T_{\alpha}T_{\alpha}^*
\big)
\right) 
\right)        \\
&+
|z_0|^2\Big(
\sum_{\alpha=1}^t
{\rm Tr}\big(
T_{\alpha}T_{\alpha}^*
\big)+{\rm Tr}\big(
Q_0Q_0^*
\big)+{\rm Tr}\big(
Q_{t+1}Q_{t+1}^*
\big)
\Big)         \\       
&+N^{-\frac{1}{2}}
\Big(
\overline{z}_0{\rm Tr}\big(
\hat{Z} Q_0Q_0^*
\big)+z_0{\rm Tr}\big(
\hat{Z}^* Q_0Q_0^*
\big)
+
{\rm Tr}\big(
\hat{Z} Q_{t+1}A_{t+1}^*Q_{t+1}^*
\big)+{\rm Tr}\big(
\hat{Z}^* Q_{t+1}A_{t+1}Q_{t+1}^*
\big)
\Big)  \\
&+\sum_{1\leq i < j\leq n}
\Big |
\Big(
\sum_{\alpha=1}^t a_{\alpha}T_{\alpha}T_{\alpha}^*
+z_0Q_0Q_0^*+Q_{t+1}A_{t+1}Q_{t+1}^*
\Big)_{i,j}
\Big|^2          \\
&-{\rm Tr}Q_{t+1}
\big(
z_0\mathbb{I}_{r_{t+1}}-A_{t+1}
\big)^*
\big(
z_0\mathbb{I}_{r_{t+1}}-A_{t+1}
\big)
Q_{t+1}^*,
\end{aligned}
\end{equation}
\end{small}
and
\begin{small}
\begin{equation}\label{gYUT}
\begin{aligned}
g(T,Y,Q)=& \Big(\det\big(
\mathbb{I}_n-\sum_{\alpha=1}^t T_{\alpha}T_{\alpha}^*
-Q_0Q_0^*-Q_{t+1}Q_{t+1}^*
\big)\Big)^{R_0-n}  \Bigg(
\det\!\begin{bmatrix}
\sqrt{\gamma_N}Z & -Y^*   \\
Y  &   \sqrt{\gamma_N} Z^*
\end{bmatrix}
\Bigg)^{R_0-n}
\\
&\times
\Big(\prod_{\alpha=1}^t\prod_{j=1}^n
\big(
t_{j,j}^{(\alpha)}
\big)^{R_{\alpha,N}-j}\Big)
\Big(\prod_{\alpha=1}^t
\big(
\det(A_{\alpha})
\big)^{R_{\alpha,N}-n} \Big) \det\!\bigg(
\widehat{L}_1+\sqrt{\gamma_N}\widehat{L}_2
\bigg),
\end{aligned}
\end{equation}
\end{small}
Here $R_{\alpha,N}$ is defined in \eqref{ralpha N} and
\begin{equation}\label{gammaN}
\gamma_N=\frac{N}{N-n},
\end{equation}
\begin{equation} \label{norm-1}
\widetilde{C}_{N,n,\tau}=
\Big( \frac{N-n}{\pi\tau} \Big)^{n^2}
\frac{1}{C_{N,\tau}}e^{-\frac{N}{\tau}\sum_{k=1}^n|z_k|^2}
\prod_{1\leq i <j \leq n}|z_i-z_j|^2
\prod_{\alpha=1}^t
\frac{\pi^{nr_{\alpha}-\frac{n(n-1)}{2}}}
{\prod_{j=1}^n(r_{\alpha}-j)!}
\end{equation}
and
\begin{equation*}
C_{N,\tau}= 
\tau^{nN-\frac{n(n-1)}{2}}
\pi^{n(r+1)} N^{-\frac{1}{2}n(n+1)} (N-n)^{-n(N-n)}     \prod_{k=N-n-r}^{N-r-1}  k !,
\end{equation*}
the matrix  variables have the restrictions
\begin{equation}\label{integrationregionprop}
\sum_{\alpha=1}^t T_{\alpha}T_{\alpha}^*+Q_0Q_0^*+Q_{t+1}Q_{t+1}^*\leq \mathbb{I}_n.
\end{equation}
\end{proposition}
For proof, we refer the reader to read Section 2.2 in \cite{LZ23} directly, which is our first paper in this series, and has the same setting with the present paper.

Now we need to give a concentration reduction by restricting integration region to a small domain. For this, we need to specify the scale of $N$ in each quantity, from  \eqref{A alpha} we have  
\begin{equation*}
A_{\alpha}=
\begin{bmatrix}
\sqrt{ \gamma_N} (z_0-a_{\alpha}) \mathbb{I}_n &
-Y^* \\
Y &
\sqrt{ \gamma_N} \overline{z_0-a_{\alpha}}\mathbb{I}_n
\end{bmatrix}+
\sqrt{ \gamma_N}N^{-\frac{1}{2}}
\begin{bmatrix}
\hat{Z}  &   \\
&   \hat{Z}^*
\end{bmatrix},
\end{equation*}
and 
\begin{small}
\begin{equation}\label{logAalpha decompose}
\log\det(A_{\alpha})=\log\det\left(
\gamma_N f_{\alpha}\mathbb{I}_n+YY^*
\right)+\log\det\left(
\mathbb{I}_{2n}+\sqrt{ \gamma_N}N^{-\frac{1}{2}}
\widehat{A}_{\alpha}
\right),
\end{equation}
\end{small}
where,   for $ \alpha=1,\cdots,t$, 
\begin{equation}\label{falpha}f_{\alpha}:=|z_0-a_{\alpha}|^2,\end{equation}
and
\begin{equation}\label{HAalpha}
\begin{aligned}
&\widehat{A}_{\alpha}=
\left[\begin{smallmatrix}
\sqrt{\gamma_N}\overline{z_0-a_{\alpha}}\left(
 \gamma_N f_{\alpha}\mathbb{I}_n+Y^*Y
\right)^{-1}\hat{Z}
& Y^*\left(
\gamma_N f_{\alpha}\mathbb{I}_n+YY^*
\right)^{-1}\hat{Z}^*       \\
-Y\left(
\gamma_N f_{\alpha}\mathbb{I}_n+Y^*Y
\right)^{-1}\hat{Z}
& \sqrt{\gamma_N}(z_0-a_{\alpha})\left(
\gamma_Nf_{\alpha}\mathbb{I}_n+YY^*
\right)^{-1}\hat{Z}^*
\end{smallmatrix}\right].
\end{aligned}
\end{equation}

Take all the  leading terms of $f(T,Y,Q)$ in  \eqref{fTY} we have
\begin{small}
\begin{equation}\label{f0TY}
\begin{aligned}
&f_0(\tau;T,Y,Q)=\sum_{\alpha=1}^t\left(
c_{\alpha}\log\det(T_{\alpha}T_{\alpha}^*) 
+c_{\alpha}
\log\det\big(
f_{\alpha}\mathbb{I}_n+YY^*
\big)\right)       
 +
\frac{1}{\tau}\Big(h_{0}(T,Q)- {\rm Tr}\big(YY^*)\Big)
\end{aligned}
\end{equation}
\end{small}
with \begin{small}
\begin{equation*}
\begin{aligned}
&h_{0}(T,Q)= 
-\sum_{\alpha=1}^t f_{\alpha}{\rm Tr}\big(
T_{\alpha}T_{\alpha}^*
\big)
-{\rm Tr}Q_{t+1}
\big(
z_0\mathbb{I}_{r_{t+1}}-A_{t+1}
\big)^*
\big(
z_0\mathbb{I}_{r_{t+1}}-A_{t+1}
\big)
Q_{t+1}^*               \\
&+
|z_0|^2\bigg(
\sum_{\alpha=1}^t
{\rm Tr}\big(
T_{\alpha}T_{\alpha}^*
\big)+{\rm Tr}\big(
Q_0Q_0^*
\big)+{\rm Tr}\big(
Q_{t+1}Q_{t+1}^*
\big)
\bigg)         \\ 
&+\sum_{1\leq i < j\leq n}
\bigg|
\bigg(
\sum_{\alpha=1}^t a_{\alpha}T_{\alpha}T_{\alpha}^*
+z_0Q_0Q_0^*+Q_{t+1}A_{t+1}Q_{t+1}^*
\bigg)_{i,j}
\bigg|^2.         
\end{aligned}
\end{equation*}
\end{small}
Note that  the two sets of matrix variables $\{Y\}$ and  $ \{Q_0, Q_{t+1}, T_{1}, \ldots, T_t\}$ in $f_0(\tau;T,Y)$ are separate 
in  $f_0(\tau;T,Y)$,  we can do Taylor expansion near  the maximum points respectively, according to  Lemma \ref{maximumY}  and Lemma  \ref{maximum lemma}  in the subsequent Section \ref{Sectmax} below. 

Take the singular value decomposition
\begin{small}
\begin{equation}\label{singularcorre}
Y=U_1\sqrt{\Lambda}U_2,\quad
\Lambda={\rm diag}\left( \lambda_1,\cdots,\lambda_n \right),\quad
\lambda_1\geq \cdots \geq \lambda_n \geq 0,
\end{equation}
\end{small}
 the Jacobian reads 
\begin{equation}\label{singularjacobian}
{\rm d}Y=\pi^{n^2}\Big( \prod_{i=1}^{n-1}i! \Big)^{-2}
\prod_{1\leq i<j\leq n}(\lambda_{j}-\lambda_{i})^2 
{\rm d}\Lambda{\rm d}U_1{\rm d}U_2,
\end{equation}
where 
$U_1$ and $U_2$ are chosen from the unitary group  $\mathcal{U}(n)$ with the Haar measure.

By restricting the  integration region to 
\begin{equation}\label{regiondelta}
\Omega_{N,\delta}=\Omega
\cap
A_{N,\delta},
\end{equation}
where
\begin{small}
\begin{equation}\label{Omega}
\Omega=\bigg\{ \big(\{T_\alpha\}, U_1,\Lambda,U_2, Q_0,Q_{t+1}\big)\Big|\sum_{\alpha=1}^t
T_{\alpha}T_{\alpha}^*+Q_0Q_0^*+Q_{t+1}Q_{t+1}^*\leq \mathbb{I}_n
,\lambda_1\geq \cdots \geq \lambda_n \geq 0 \bigg\}
\end{equation}
\end{small}
and
\begin{equation}\label{ANdelta}
A_{N,\delta}=\left\{ \big(\{T_\alpha\},U_1,\Lambda,U_2, Q_0,Q_{t+1}\big)\Big|
S(T,Y,Q)
\leq \delta \right\},
\end{equation}
with positive constant $\delta$ and
\begin{small}
\begin{equation}\label{SYUT}
\begin{aligned}
S(T,Y,Q)&=\sum_{\alpha=1}^t
\Big(
{\rm Tr}\big(
T_{{\rm d},\alpha}-\tau c_{\alpha} (|z_0-a_{\alpha}|^{2}+t_0)^{-1}
\mathbb{I}_n
\big)
\big(
T_{{\rm d},\alpha}-\tau c_{\alpha} (|z_0-a_{\alpha}|^{2}+t_0)^{-1}\mathbb{I}_n
\big)^*
\\&+{\rm Tr}(T_{{\rm u},\alpha}T_{{\rm u},\alpha}^*)
\Big)+
{\rm Tr}\Big(
Q_0Q_0^*+Q_{t+1}Q_{t+1}^*+
\big(
\Lambda-t_0\mathbb{I}_n
\big)^2\Big),
\end{aligned}
\end{equation}
\end{small}
we have  a concentration reduction, the proof is given in Section \ref{propfo}.    
\begin{proposition}\label{RNn delta}
 With the same notations and conditions as in Theorem \ref{2-complex-correlation} and Proposition \ref{foranlysis}, then for any positive constant $\delta$ there exists $\Delta>0$ such that
\begin{equation}\label{RNndelta}
R_N^{(n)}(A_0;z_1,\cdots,z_n)=
D_{N,n}\,\Big(
I_{N,\delta}+O\big(
e^{-\frac{1}{2}N\Delta}
\big)
\Big),
\end{equation}
where
\begin{small}
\begin{equation}\label{INdelta}
\begin{aligned}
&I_{N,\delta}=
\pi^{n^2}\Big( \prod_{i=1}^{n-1}i! \Big)^{-2}  \int_{\Omega_{N,\delta}} 
  g(T,U_1\sqrt{\Lambda}U_2,Q)\prod_{1\leq i<j\leq n}(\lambda_{j}-\lambda_{i})^2 
\\&\times  \exp\Big\{ 
  N\big(f(T,U_1\sqrt{\Lambda}U_2,Q)
  -f_0\big(\tau;\big\{\frac{\tau c_{\alpha}}{f_{\alpha}+t_0}\mathbb{I}_n\big\},\sqrt{t_0}U_1U_2,0\big)\big)
 \Big\} {\rm d}V,
 \end{aligned}
\end{equation}
\end{small}
  and 
\begin{equation}\label{DNn}
{\rm d}V= {\rm d}\Lambda{\rm d}U_1{\rm d}U_2
{\rm d}Q_0{\rm d}Q_{t+1} \prod_{\alpha=1}^t
{\rm d}T_{\alpha},\quad D_{N,n}=\widetilde{C}_{N,n}
e^{nN(|z_0|^2/\tau-1)}\prod_{\alpha=1}^t
(\tau c_{\alpha})^{nNc_{\alpha}}.
\end{equation}

\end{proposition}

The remaining task is to analyze the matrix integral $I_{N,\delta}$ with sufficiently small $\delta>0$, which can be followed by a quite standard procedure of Laplace method. To be precise, the proof will be divided into four steps: {\bf Step 1} deals with Taylor expansion of $f(T,Y,Q)$ and the integration region(cf \eqref{integrationregionprop}), The most relevant variables are triangular matrices $\big\{ T_{\alpha} \big\}$, equivalently, diagonal matrices $\big\{ T_{{\rm d},\alpha} \big\}$ and strictly upper triangular matrices $\big\{ T_{{\rm u},\alpha} \big\}$ as in \eqref{Talpha 2}. With the coming of the change of variables
\begin{equation*}
T_{{\rm d},1},T_{{\rm d},2},\cdots,T_{{\rm d},t}\longmapsto
\widetilde{T}_{{\rm d},1},\widetilde{T}_{{\rm d},2},\cdots,\widetilde{T}_{{\rm d},t}\longmapsto \Gamma_1,\widetilde{T}_{{\rm d},2},\cdots,\widetilde{T}_{{\rm d},t}
\end{equation*}
and 
\begin{equation*}
T_{{\rm u},1},T_{{\rm u},2},\cdots,T_{{\rm u},t}\longmapsto
G_1,T_{{\rm u},2},\cdots,T_{{\rm u},t};
\end{equation*}
{\bf Step 2} deals with Taylor expansion of $g(T,Y,Q)$, which is totally based on elementary calculations; {\bf Step 3} is a standard procedure in Laplace method to remove the error terms resulted from the expansions in first two steps; In {\bf Step 4}, after dealing with several matrix integrals, we can complete the proof.

\subsection{Maximum  Lemmas} \label{Sectmax}
Given   complex numbers $z_0$,       $a_{\alpha}$  and   
$c_{\alpha}>0 $    ($\alpha=1,\cdots,t$),     assume that 
\begin{small}
\begin{equation*} 
\sum_{\alpha=1}^t c_{\alpha}=1, \quad  \sum_{\alpha=1}^t \frac{c_{\alpha}}{|z_0-a_{\alpha}|^2}> \frac{1}{\tau}.
 \end{equation*}
\end{small}
Let $t_0$ be  the unique  positive  solution of the equation   
\begin{small}
\begin{equation*}
\sum_{\alpha=1}^t \frac{c_{\alpha}}{|z_0-a_{\alpha}|^2+t_0}=\frac{1}{\tau}.
\end{equation*}
\end{small}

The following two lemmas hold and play a  significant  role in investigating local eigenvalue statistics of the deformed GinUEs. See also \cite[Lemma 2.4 and Lemma 2.5]{LZ23}. 
\begin{lemma}\label{maximumY}
As a function of a non-negative definite matrix $H$,
\begin{equation*}
 \sum_{\alpha=1}^t c_{\alpha}
\log\det\big(
|z_0-a_{\alpha}|^2 \mathbb{I}_n+H
\big)   -
\frac{1}{\tau}  {\rm Tr}(H)
\end{equation*}
attains it maximum only at  $H=t_0 I_n$. \end{lemma}

\begin{proof}
Obviously, it's sufficient to check   the $n=1$ case. For this, take the first and second derivatives,  find the saddle point and we can  easily complete the proof.   
\end{proof}

\begin{lemma}\label{maximum lemma}
Assume that  $z_0$ is not  an eigenvalue of   the   $n\times n$   complex normal   matrix  $D$,  let 
\begin{small}
\begin{equation*}
\begin{aligned}
J_n&=\sum_{\alpha=1}^t\Big(
\tau c_{\alpha}\log\det\big(
T_{\alpha}T_{\alpha}^*
\big)-|z_0-a_{\alpha}|^2{\rm Tr}\big(
T_{\alpha}T_{\alpha}^*
\big)
\Big)
+|z_0|^2 {\rm Tr}\Big(
\sum_{\alpha=1}^t T_{\alpha}T_{\alpha}^*
+AA^*+BB^*
\Big)
\\&-{\rm Tr}\Big(
B\big(
z_0\mathbb{I}_{l_2}-D
\big)^*\big(
z_0\mathbb{I}_{l_2}-D
\big)B^*
\Big)
+\sum_{1\leq i<j\leq n}\left|
\Big(
\sum_{\alpha=1}^t a_{\alpha}T_{\alpha}T_{\alpha}^*
+z_0AA^*+BDB^*
\Big)_{i,j}
\right|^2,
\end{aligned}
\end{equation*}
\end{small}
where  as matrix-valued  variables
$T_{\alpha}$'s  are  $n\times n$  upper-triangular matrices  with  positive diagonal elements,   
$A$ and $B$  are   $n\times l_1$ and  $n\times l_2$  matrices respectively. Then 
\begin{equation*}
J_n\leq n\sum_{\alpha=1}^t \tau c_{\alpha}
\log\frac{\tau c_{\alpha}}{|z_0-a_{\alpha}|^2+t_0}
+n(t_0+|z_0|^2-\tau),
\end{equation*} 
whenever  
\begin{equation*}
\sum_{\alpha=1}^t T_{\alpha}T_{\alpha}^*+AA^*+BB^*
\leq \mathbb{I}_n,.
\end{equation*}
Moreover, the  equality holds  if and only if 
\begin{equation*} 
A=0,\quad
B=0, \quad T_{\alpha}=\sqrt{\frac{\tau c_{\alpha}}{|z_0-a_{\alpha}|^2+t_0}}
\mathbb{I}_n, \quad \alpha=1,  \ldots, t.
\end{equation*} 
\end{lemma}
For proof see \cite[Lemma 2.5]{LZ23}.

\section{Proofs of Proposition \ref{RNn delta} and Theorem \ref{2-complex-correlation}}  \label{proofs}
\subsection{Proof of Proposition \ref{RNn delta}}\label{propfo}
\begin{proof}[Proof of Proposition \ref{RNn delta}.] 

\hspace*{\fill}

By  Lemma \ref{maximumY}  and Lemma  \ref{maximum lemma}  in Section \ref{Sectmax} below,   we see that $f_0(\tau;T,Y,Q)$ in \eqref{f0TY}  attains its maximum  only at  
\begin{equation}\label{maximumpoint}
Y=\sqrt{t_0}U_1U_2, \ Q_0=0, \ Q_{t+1}=0, \quad T_{\alpha}=\sqrt{\tau c_{\alpha}/(f_{\alpha}+t_0})\mathbb{I}_n,\quad \alpha=1,\cdots,t. 
\end{equation}
So  there exists $\Delta>0$ such that 
\begin{equation}\label{f0TY minus f0}
f_0(\tau;T,Y,Q)-f_0\big(\tau;\big\{\frac{\tau c_{\alpha}}{f_{\alpha}+t_0}\mathbb{I}_n\big\},\sqrt{t_0}U_1U_2,0\big)\leq -\Delta
\end{equation}
in  the domain $\Omega\cap
A_{N,\delta }^{c}$. Here it should be noted that we have taken $Y$ to be its representation of singular value decomposition \eqref{singularcorre}, while for simplicity we still use the symbol $Y$ in the following argument.

On one hand,  noting that  $$ \sum_{\alpha=1}^t
T_{\alpha}T_{\alpha}^*+Q_0Q_0^*+Q_{t+1}Q_{t+1}^*\leq \mathbb{I}_n$$
in the domain   $\Omega$ defined by  \eqref{Omega}, 
we know  for 
$\hat{Z}$
in a compact subset of $\mathbb{C}^n$ that  
 the partial  sub-leading terms of $f(T,Y,Q)$ in  \eqref{fTY} can be   bounded  by  $O(N^{-\frac{1}{2}})$, that is, 
 \begin{small}
\begin{equation}\label{error term1}
\begin{aligned}
&N^{-\frac{1}{2}}\sum_{\alpha=1}^t 
\left(
\overline{a}_{\alpha}{\rm Tr}\big(
\hat{Z} T_{\alpha}T_{\alpha}^*
\big)+a_{\alpha}{\rm Tr}\big(
\hat{Z}^* T_{\alpha}T_{\alpha}^*
\big)
\right) 
+N^{-\frac{1}{2}}
\left(
\overline{z}_0{\rm Tr}\big(
\hat{Z} Q_0Q_0^*
\big)+z_0{\rm Tr}\big(
\hat{Z}^* Q_0Q_0^*
\big)\right)
\\
&+N^{-\frac{1}{2}}
\left(
{\rm Tr}\big(
\hat{Z} Q_{t+1}A_{t+1}^*Q_{t+1}^*
\big)+{\rm Tr}\big(
\hat{Z}^* Q_{t+1}A_{t+1}Q_{t+1}^*
\big)
\right)
=O(N^{-\frac{1}{2}}).
\end{aligned}
\end{equation}
\end{small}
On the other hand,  apply the singular value decomposition \eqref{singularcorre} for $Y$ and note that in \eqref{falpha} all $f_{\alpha}\not=0$ because we have assumed that $a_{\alpha}\not=z_0$ for $\alpha=1,\cdots,t$ in Theorem \ref{2-complex-correlation}, we obtain 
\begin{small}
\begin{equation*}\label{HAalpha analysis}
\begin{aligned}
&(
\sqrt{\gamma_N}f_{\alpha}\mathbb{I}_n+Y^*Y)^{-1}=O(1),\quad Y(\sqrt{\gamma_N}f_{\alpha}\mathbb{I}_n+Y^*Y)^{-1}=O(1),
\end{aligned}
\end{equation*}
\end{small}
and 
\begin{small}
\begin{equation}\label{error term2}
\log\det\left(
\mathbb{I}_{2n}+\sqrt{\gamma_N} N^{-\frac{1}{2}}
\widehat{A}_{\alpha}
\right)=O(N^{-\frac{1}{2}})
\end{equation}
\end{small}
since $\widehat{A}_{\alpha}=O(1)$. 
Meanwhile, 
\begin{small}
\begin{equation*}
\log\det\left(
\gamma_N f_{\alpha}\mathbb{I}_{n}+
YY^*
\right)=\log\det\left(
f_{\alpha}\mathbb{I}_{n}+
YY^*
\right)+ O\big( N^{-1} \big).
\end{equation*}
\end{small}
Combination of  \eqref{fTY}, \eqref{logAalpha decompose}, \eqref{f0TY}, \eqref{error term1} and \eqref{error term2}  give rise to 
\begin{small}
\begin{equation}\label{fTY minus f00}
\begin{aligned}
&N\Big(
 f(T,Y,Q)-f_0\big(\tau;\big\{\frac{\tau c_{\alpha}}{f_{\alpha}+t_0}\mathbb{I}_n\big\},\sqrt{t_0}U_1U_2,0\big)
\Big)=
\\
&N\Big(
f_0(\tau;T,Y,Q)-f_0\big(\tau;\big\{\frac{\tau c_{\alpha}}{f_{\alpha}+t_0}\mathbb{I}_n\big\},\sqrt{t_0}U_1U_2,0\big)
\Big)+\frac{n}{\tau}{\rm Tr}(\Lambda)+O\big(
\sqrt{N}
\big).
\end{aligned}
\end{equation}
\end{small}

For $g(T,Y,Q)$ in \eqref{gYUT},  all factors can be controlled by a  polynomial function, so by taking a sufficiently large $N_0>\sum_{\alpha=1}^t (n-R_{\alpha,N})/c_{\alpha}$ we know 
\begin{small}
\begin{equation*}
\begin{aligned}
&\int_{A_{N,\delta}^{c}\cap \Omega}
| g(T,Y,Q) |
\prod_{1\leq i<j\leq n}(\lambda_j-\lambda_i)^2
\\&\times
\exp\left\{ N_0\Big(
 f_0(\tau;T,Y,Q)-f_0\big(\tau;\big\{\frac{\tau c_{\alpha}}{f_{\alpha}+t_0}\mathbb{I}_n\big\},\sqrt{t_0}U_1U_2,0\big)
\Big)+\frac{n}{\tau}{\rm Tr}(\Lambda) \right\}
{\rm d}V
=O(1).
\end{aligned}
\end{equation*}
\end{small}
Hence,
\begin{small}
\begin{equation*}
\begin{aligned}
&\left|
\int_{A_{N,\delta}^{\complement}\cap \Omega}
g(T,Y,Q)
\prod_{1\leq i<j\leq n}(\lambda_j-\lambda_i)^2
 \exp\left\{ N\Big(
 f(T,Y,Q)-f_0\big(\tau;\big\{\frac{\tau c_{\alpha}}{f_{\alpha}+t_0}\mathbb{I}_n\big\},\sqrt{t_0}U_1U_2,0\big)
\Big) \right\}
{\rm d}V
\right|   \\
&\leq e^{-(N-N_0)\Delta+O(\sqrt{N})}
 O(1)    =O\big(
e^{-\frac{1}{2}N\Delta}
\big).
\end{aligned}
\end{equation*}
\end{small}

This  thus completes the proof.
\end{proof}

 \subsection{Proof of Theorem \ref{2-complex-correlation}}\label{theorempfo}
The goal in this section is to complete the proof of Theorem \ref{2-complex-correlation}. With Section \ref{sectnotation} in mind, we still need some extra notations: For a family of matrices $\{ T_{{\rm u},\alpha}|\alpha=1,\cdots,t\}$, define 
     $\|T_{\rm u}\|:=\sum_{\alpha=1}^t \|T_{{\rm u},\alpha}\|,\quad    
     \| T_{\rm u}\|_{2} :=\sum_{\alpha=2}^t
 \| T_{{\rm u},\alpha}\|.$ 
 Similarly,  For a family of matrices $\{ \widetilde{T}_{{\rm d},\alpha}|\alpha=1,\cdots,t\}$,
$ 
  \|\widetilde{T}_{{\rm d}}\|:=\sum_{\alpha=1}^t
 \| \widetilde{T}_{{\rm d},\alpha}\|, \quad  
 \| 
  \widetilde{T}_{{\rm d}}\|_{2} :=\sum_{\alpha=2}^t
 \| 
  \widetilde{T}_{{\rm d},\alpha}\|.
$ For a complex matrix $M$, denote $M^{({\rm d})}$ the diagonal part of M, and $M^{({\rm off},{\rm u})}$ the strictly upper triangular part of M; If the object is $M_{\alpha}$ with index $\alpha$, then denote the corresponding quantity as $M^{(\alpha,{\rm d})}$ and $M^{(\alpha,{\rm off},{\rm u})}$.

In view of Proposition \ref{RNn delta}, we will restrict all the matrix variables in the region $\Omega_{N,\delta}$ defined in \eqref{regiondelta} for sufficiently small $\delta>0$. Here we must emphasize that $\delta$ is fixed independent of $N$, while we need $\delta>0$ to be sufficiently small so that Taylor expansions of the relevant matrix variables are permitted. As the number of relevant matrix variables is certainly finite,  the required $\delta$ can always be realized. In addition, in the following analysis, all error terms are written as the form $O(A)$, 
which means the error terms are bounded by $CA$ for some positive constant $C$ which is independent of $N$.

{\bf Step 1: Taylor expansion of $f(T,Y,Q)$.} 
According to \eqref{INdelta}, on the exponent, for the term \begin{small}
$f(T,U_1\sqrt{\Lambda}U_2,Q)
  -f_0\big(\tau;\big\{\frac{\tau c_{\alpha}}{f_{\alpha}+t_0}\mathbb{I}_n\big\},\sqrt{t_0}U_1U_2,0\big)$\end{small}, when we have done Taylor expansions of the relevant matrix variables, we need to do appropriate rescaling to cancel out the factor $N$, see \eqref{change scale}. Accordingly we introduce the notation of "typical size" of each term, by means of the rescaling. For example, in \eqref{change scale}, we make the change of variables \begin{small}
 $ (Q_0,Q_{t+1})\rightarrow
 N^{-\frac{1}{2}}(Q_0,Q_{t+1})$
  \end{small}, so the typical sizes of $\|Q_0\|$ and $\|Q_{t+1}\|$ are $N^{-\frac{1}{2}}$, and the typical size of $\|Q_0\|\|Q_{t+1}\|$ is $N^{-1}$. {\bf The terms of typical size larger than or equal to $N^{-1}$ should be reserved, otherwise the terms of typical size smaller than or equal to $N^{-\frac{3}{2}}$ should be considered as error terms}, which can be cancelled out throughout standard method. 

With the notation $f_{\alpha}$ in \eqref{falpha}   
 noticing the  integration domain and   \eqref{SYUT}    the decomposition of     $T_{\alpha}$   as  sum of  a diagonal matrix $\sqrt{T_{{\rm d},\alpha}}$ and a strictly upper triangular  matrix ${T_{{\rm u},\alpha}}$ as in \eqref{Talpha 2}, 
  introduce new matrix variables 
\begin{equation}\label{Tdalpha1}
T_{{\rm d},\alpha}=\frac{\tau c_{\alpha}}{f_{\alpha}+t_0}
\mathbb{I}_n+\widetilde{T}_{{\rm d},\alpha},
\end{equation}
It's easy to obtain
\begin{equation}\label{Talphaexpansion1}
\begin{aligned}
c_{\alpha}{\rm Tr}\log T_{{\rm d},\alpha}
&-\frac{1}{\tau} f_{\alpha}{\rm Tr}(T_{{\rm d},\alpha})
=
nc_{\alpha}\big(
\log \frac{\tau c_{\alpha}}{f_{\alpha}+t_0}-
\frac{f_{\alpha}}{f_{\alpha}+t_0}
\big)
\\
&+\frac{1}{\tau}t_0{\rm Tr}\big(\widetilde{T}_{{\rm d},\alpha}\big)
-\frac{(f_{\alpha}+t_0)^2}{2\tau^2 c_{\alpha}}{\rm Tr}\big(
\widetilde{T}_{{\rm d},\alpha}^2\big)
+O\big(
\|
\widetilde{T}_{{\rm d},\alpha}
\|^3
\big),
\end{aligned}
\end{equation}
\begin{equation*}
\sqrt{T_{{\rm d},\alpha}}=\sqrt{\frac{\tau c_{\alpha}}{f_{\alpha}+t_0}}
\left(
\mathbb{I}_n+\frac{f_{\alpha}+t_0}{2\tau c_{\alpha}}\widetilde{T}_{{\rm d},\alpha}
+O\big(
\|
\widetilde{T}_{{\rm d},\alpha}
\|^2
\big)
\right),
\end{equation*}
\begin{equation*}
\begin{aligned}
T_{\alpha}T_{\alpha}^*
&=\frac{\tau c_{\alpha}}{f_{\alpha}+t_0}\mathbb{I}_n+\widetilde{T}_{{\rm d},\alpha}+
\sqrt{\frac{\tau c_{\alpha}}{f_{\alpha}+t_0}}
\left(
\mathbb{I}_n+\frac{f_{\alpha}+t_0}{2\tau c_{\alpha}}\widetilde{T}_{{\rm d},\alpha}
+O\big(
\|
\widetilde{T}_{{\rm d},\alpha}
\|^2
\big)
\right)
T_{{\rm u},\alpha}^*     
\\
&+\sqrt{\frac{\tau c_{\alpha}}{f_{\alpha}+t_0}}
T_{{\rm u},\alpha}
\left(
\mathbb{I}_n+\frac{f_{\alpha}+t_0}{2\tau c_{\alpha}}\widetilde{T}_{{\rm d},\alpha}
+O\big(
\|
\widetilde{T}_{{\rm d},\alpha}
\|^2
\big)\right)
+T_{{\rm u},\alpha}T_{{\rm u},\alpha}^*,
\end{aligned}
\end{equation*}
from which we have for $i<j$,
\begin{equation}\label{Talphaexpansion4}
\left(
\sum_{\alpha=1}^t a_{\alpha}T_{\alpha}T_{\alpha}^*
\right)_{i,j}
=\sum_{\alpha=1}^t a_{\alpha}\sqrt{\frac{\tau c_{\alpha}}{f_{\alpha}+t_0}}
t_{i,j}^{(\alpha)}
+O\left(
\| T_{{\rm u},\alpha} \|^2+\| T_{{\rm u},\alpha} \|
\| \widetilde{T}_{{\rm d},\alpha} \|
\right),
\end{equation} 
and
\begin{equation}\label{Talphaexpansion5}
{\rm Tr}\big(
\hat{Z}T_{\alpha}T_{\alpha}^*
\big)
=
\frac{\tau c_{\alpha}}{f_{\alpha}+t_0}
{\rm Tr}\big(
\hat{Z}
\big)+
{\rm Tr}\big(
\hat{Z}\widetilde{T}_{{\rm d},\alpha}
\big)+O\left(
\|
T_{{\rm u},\alpha}
\|^2
\right).
\end{equation}
At the same time,
\begin{equation*}
\left(
z_0Q_0Q_0^*+Q_{t+1}A_{t+1}Q_{t+1}^*
\right)_{i,j}
=
O\left(
\| Q_0 \|^2+
\| Q_{t+1} \|^2
\right).
\end{equation*}
Moreover, we see from the condition $\sum_{\alpha=1}^t \frac{\tau c_{\alpha}}{f_{\alpha}+t_0}=1$ that  the restriction condition in  \eqref{Omega} becomes
\begin{equation*}
\sum_{\alpha=1}^t \widetilde{T}_{{\rm d},\alpha}+
\sum_{\alpha=1}^t \sqrt{\frac{\tau c_{\alpha}}{f_{\alpha}+t_0}}
T_{{\rm u},\alpha}+
\sum_{\alpha=1}^t \sqrt{\frac{\tau c_{\alpha}}{f_{\alpha}+t_0}}
T_{{\rm u},\alpha}^*+S\leq 0,
\end{equation*}
where
\begin{equation}\label{S}
\begin{aligned}
S&=\sum_{\alpha=1}^t T_{{\rm u},\alpha}T_{{\rm u},\alpha}^*
+Q_0Q_0^*+Q_{t+1}Q_{t+1}^*   \\
&+\sum_{\alpha=1}^t T_{{\rm u},\alpha}\Big(
\sqrt{T_{{\rm d},\alpha}}-
\sqrt{\frac{\tau c_{\alpha}}{f_{\alpha}+t_0}}\mathbb{I}_n
\Big)+
\sum_{\alpha=1}^t \Big(
\sqrt{T_{{\rm d},\alpha}}-
\sqrt{\frac{\tau c_{\alpha}}{f_{\alpha}+t_0}}\mathbb{I}_n
\Big)T_{{\rm u},\alpha}^*,
\end{aligned}
\end{equation}
the first line contains all diagonal parts of $S$, while the second line only contains the non-diagonal parts.

 Let  $S^{({\rm d})}$ be a  diagonal  matrix extracted   from the diagonal part of $S$ given  in \eqref{S},  introduce a new diagonal matrix   $\Gamma_1$ and make change of variables  from 
 $\widetilde{T}_{{\rm d},1}, \widetilde{T}_{{\rm d},2}, \cdots, \widetilde{T}_{{\rm d},t}$ to 
\begin{equation}\label{matrix transformations1}
\begin{aligned}
&\Gamma_1:=\sum_{\alpha=1}^t    \widetilde{T}_{{\rm d},\alpha}+S^{({\rm d})},
\quad \widetilde{T}_{{\rm d},2},\, \ldots, \widetilde{T}_{{\rm d},t}.
\end{aligned}
\end{equation}
Note that $S^{({\rm d})}$ doesn't depend on $\widetilde{T}_{{\rm d},\alpha}$, so this transformation is linear and
\begin{equation}\label{Sdexpan}
S^{({\rm d})}=
O\big(
\| T_{{\rm u}} \|^2+\| Q_{t+1} \|^2+\| Q_0 \|^2
\big).
\end{equation}
We will see that  the typical size for  $Q_0, Q_{t+1}, \widetilde{T}_{{\rm d},2},\, \ldots, \widetilde{T}_{{\rm d},t}$ are $N^{-1/2}$   while it  is  $N^{-1}$ for  $\Gamma_1$. See \eqref{change scale} for the exact meaning. 

Also let  $S^{({\rm off},{\rm u})}$ be a strictly upper triangular matrix extracted from the strictly upper triangular part of $S$ given  in \eqref{S},  introduce a new strictly upper triangular matrix   $G_1$ and make change of variables  from 
 $T_{{\rm u},1}, T_{{\rm u},2}, \cdots, T_{{\rm u},t}$ to 
\begin{equation}\label{matrix transformations2}
\begin{aligned}
&G_1:=\sum_{\alpha=1}^t  \sqrt{\frac{\tau c_{\alpha}}{f_{\alpha}+t_0}}
T_{{\rm u},\alpha}+S^{({\rm off},{\rm u})},
\quad T_{{\rm u},2},\, \ldots,T_{{\rm u},t}.
\end{aligned}
\end{equation}
We will see that  the typical size for $T_{{\rm u},2}, \cdots, T_{{\rm u},t}$  are $N^{-1/2}$   while it  is  $N^{-1}$ for  $G_1$. See \eqref{change scale} for the exact meaning. So we will do Taylor expansions up to some proper orders  accordingly.

As to the change of variables   \eqref{matrix transformations2}, the Jacobian determinant reads
\begin{equation*}
\det\Big(
\frac{\partial T_{{\rm u},1}}{\partial G_1}
\Big)
=\left(
\frac{f_1+t_0}{\tau c_1}
\right)^{\frac{n(n-1)}{2}}+
O\big(
\| T_{{\rm u},1} \|+\| \widetilde{T}_{{\rm d}} \|
\big).
\end{equation*}
At the same time, 
\begin{equation}\label{Tdestimate0}
\| \widetilde{T}_{{\rm d}} \|=O\big(
\| \widetilde{T}_{{\rm d}} \|_2+\|\Gamma_1\|+\| Q_0 \|
+\| T_{{\rm u}} \|+\| Q_{t+1} \|
\big).
\end{equation}
Here exactly, from \eqref{matrix transformations1} and \eqref{Sdexpan} we should have
\begin{equation*}
\| \widetilde{T}_{{\rm d}} \|=O\big(
\| \widetilde{T}_{{\rm d}} \|_2+\|\Gamma_1\|+\| Q_0 \|^2
+\| T_{{\rm u}} \|^2+\| Q_{t+1} \|^2
\big),
\end{equation*}
while as \eqref{ANdelta} and \eqref{INdelta} we have
\begin{equation}\label{deltacontrol}
 \| T_{{\rm u}} \|+\| \widetilde{T}_{{\rm d}} \|+\| Q_0 \|+
\| Q_{t+1} \|
=O\big( \sqrt{\delta} \big).
\end{equation}
As $\delta$ is fixed \eqref{Tdestimate0} is right. Moreover, \eqref{Tdestimate0} already is sufficient for our analysis. From 
\begin{equation*}
S^{({\rm off},{\rm u})}=O\Big(
\| T_{{\rm u}} \|\big(
\| T_{{\rm u}} \|+\| \widetilde{T}_{{\rm d}} \|
\big)+\| Q_{t+1} \|^2+\| Q_0 \|^2
\Big),
\end{equation*}
we have
\begin{equation*}
\| T_{{\rm u}} \|=O\big(
\| G_1 \|+\|T_{{\rm u}}\|_2+\| Q_0 \|
+\| Q_{t+1} \|
\big)+\| T_{{\rm u}} \|O\big(
\| T_{{\rm u}} \|+\| \widetilde{T}_{{\rm d}} \|
\big).
\end{equation*}
From \eqref{deltacontrol} and as $\delta>0$ is sufficiently small, we have
 \begin{equation}\label{Tuestimate}
\| T_{{\rm u}} \|=O\big(
\| G_1 \|+\|T_{{\rm u}}\|_2+\| Q_0 \|
+\| Q_{t+1} \|
\big).
\end{equation}
Substitute \eqref{Tuestimate} into \eqref{Tdestimate0}, we have
\begin{equation*}
\| \widetilde{T}_{{\rm d}} \|=O\big(
\Omega_{{\rm error}}
\big),
\end{equation*}
where
\begin{equation}\label{Omegaerror}
\Omega_{{\rm error}}:=\| \widetilde{T}_{{\rm d}} \|_2+\|\Gamma_1\|
+\| G_1 \|+\|T_{{\rm u}}\|_2
+\| Q_0 \|+\| Q_{t+1} \|,
\end{equation}
and further
\begin{small}
\begin{equation}\label{Jacobiandet}
\det\Big(
\frac{\partial T_{{\rm u},1}}{\partial G_1}
\Big)
=\left(
\frac{f_1+t_0}{\tau c_1}
\right)^{\frac{n(n-1)}{2}}+O(\Omega_{{\rm error}}).
\end{equation}
\end{small}
It should be noted that $\Omega_{{\rm error}}$ plays a central role in representation of the error terms, according to \eqref{change scale}, its typical size is $N^{-\frac{1}{2}}$.

In view of \eqref{fTY} and \eqref{hQrewrite},
\begin{small}
\begin{equation}\label{FTYlinear1}
\frac{|z_0|^2}{\tau}\Big(
\sum_{\alpha=1}^t {\rm Tr}(T_{\alpha}T_{\alpha}^*)+
{\rm Tr}(Q_{0}Q_{0}^*)+{\rm Tr}(Q_{t+1}Q_{t+1}^*)
\Big)=\frac{n|z_0|^2}{\tau}+\frac{|z_0|^2}{\tau}{\rm Tr}(\Gamma_1).
\end{equation}
\end{small}
Also in view of \eqref{Talphaexpansion1}, 
\begin{small}
\begin{equation}\label{FTYlinear2}
\frac{t_0}{\tau}\sum_{\alpha=1}^t 
{\rm Tr}\big(\widetilde{T}_{{\rm d},\alpha}\big)=
\frac{t_0}{\tau}{\rm Tr}(\Gamma_1)
-\frac{t_0}{\tau}{\rm Tr}(S^{({\rm d})}).
\end{equation}
\end{small}
From \eqref{S}, we have
\begin{equation}\label{FTYlinear3}
-\frac{t_0}{\tau}{\rm Tr}(S^{({\rm d})})=
-\frac{t_0}{\tau}\sum_{\alpha=1}^t 
{\rm Tr}\big(T_{{\rm u},\alpha}T_{{\rm u},\alpha}^*\big)
-\frac{t_0}{\tau}{\rm Tr}\big( Q_0Q_0^* \big)
-\frac{t_0}{\tau}{\rm Tr}\big( Q_{t+1}Q_{t+1}^* \big).
\end{equation}

In view of \eqref{fTY}, \eqref{hQrewrite} and \eqref{Talphaexpansion4}, from transformations \eqref{matrix transformations1} and \eqref{matrix transformations2} we have
\begin{small}
\begin{equation}\label{triangle1}
\begin{aligned}
&-\sum_{\alpha=1}^t \frac{f_{\alpha}+t_0}{\tau}
{\rm Tr}\big(
T_{{\rm u},\alpha}T_{{\rm u},\alpha}^*
\big)+\frac{1}{\tau}
\sum_{1\leq i < j\leq n}
 \left|
 \left(
 \sum_{\alpha=1}^t a_{\alpha}T_{\alpha}T_{\alpha}^*
+z_0
Q_0Q_0^*+
Q_{t+1}A_{t+1}Q_{t+1}^*
 \right)_{i,j} 
 \right|^2       \\
 &=-\frac{(f_1+t_0)^2}{\tau^2 c_1}
 {\rm Tr}\left(
\sum_{\alpha=2}^t
 \sqrt{\frac{\tau c_{\alpha}}{f_{\alpha}+t_0}}T_{{\rm u},\alpha}
 \right)
 \left(
 \sum_{\alpha=2}^t
 \sqrt{\frac{\tau c_{\alpha}}{f_{\alpha}+t_0}}T_{{\rm u},\alpha}
 \right)^*  -\sum_{\alpha=2}^t
\frac{f_{\alpha}+t_0}{\tau}{\rm Tr}\big(
T_{{\rm u},\alpha}T_{{\rm u},\alpha}^*
\big)             \\
&     
+
\frac{1}{\tau}
\sum_{1\leq i < j\leq n}
 \left|
\sum_{\alpha=2}^t
\sqrt{\frac{\tau c_{\alpha}}{f_{\alpha}+t_0}}
(a_{\alpha}-a_1)
t_{i,j}^{(\alpha)}
 \right|^2+O(\Omega_{{\rm error}}^3+\|G_1\|\Omega_{{\rm error}}),
\end{aligned}
\end{equation}
\end{small}
and
\begin{small}
\begin{equation}\label{diagonal1}
\begin{aligned}
&-\sum_{\alpha=1}^t\frac{(f_{\alpha}+t_0)^2}{2\tau^2 c_{\alpha}}
{\rm Tr}
\left(
\widetilde{T}_{{\rm d},\alpha}^2
\right)
+\frac{|z_0|^2}{\tau}{\rm Tr}(\Gamma_1)
+\frac{N^{-\frac{1}{2}}}{\tau}
\sum_{\alpha=1}^t\left(
\overline{a}_{\alpha}{\rm Tr}\big(
\hat{Z}\widetilde{T}_{{\rm d},\alpha}
\big)
+a_{\alpha}{\rm Tr}\big(
\hat{Z}^*\widetilde{T}_{{\rm d},\alpha}
\big)
\right)+\frac{t_0}{\tau}{\rm Tr}(\Gamma_1)             \\
&=\frac{|z_0|^2+t_0}{\tau}{\rm Tr}(\Gamma_1)
-\frac{(f_{1}+t_0)^2}{2\tau^2 c_{1}}{\rm Tr}
\left(
 \sum_{\alpha=2}^t
\widetilde{T}_{{\rm d},\alpha}
 \right)^2
 -\sum_{\alpha=2}^t\frac{(f_{\alpha}+t_0)^2}{2\tau^2 c_{\alpha}}
{\rm Tr}\big(\widetilde{T}_{{\rm d},\alpha}^2  \big)  
   \\
&+\frac{N^{-\frac{1}{2}}}{\tau}\sum_{\alpha=2}^t
{\rm Tr}\Big(
\big(\overline{a_{\alpha}-a_1}\hat{Z}+
(a_{\alpha}-a_1)\hat{Z}^*\big)
 \widetilde{T}_{{\rm d},\alpha}
\Big)+O\Big(
\big(
\Omega_{{\rm error}}+N^{-\frac{1}{2}}
\big)
\big(
\Omega_{{\rm error}}^2+\|\Gamma_1\|
\big)
\Big).
\end{aligned}
\end{equation}
\end{small}
Also, the integral region now becomes simply
\begin{equation}\label{SimpleRegion}
\Gamma_1+G_1+G_1^*\leq 0.
\end{equation}

In view of \eqref{logAalpha decompose},
\begin{small}
\begin{equation}\label{logdetHA}
\begin{aligned}
\log\det\left(
\mathbb{I}_{2n}+\sqrt{\gamma_N}
N^{-\frac{1}{2}}\widehat{A}_{\alpha}
\right)
=N^{-\frac{1}{2}}{\rm Tr}\big(
\widehat{A}_{\alpha}
\big)-\frac{1}{2}N^{-1}{\rm Tr}\big(
\widehat{A}_{\alpha}^2
\big)
+O\big(
N^{-\frac{3}{2}}
\big).
\end{aligned}
\end{equation}
\end{small}
 Now start from \eqref{singularcorre}, we need to take a shift
\begin{equation}\label{tildeLambda}
\widetilde{\Lambda}=
\Lambda-t_0\mathbb{I}_n={\rm diag}\left( \widetilde{\lambda}_1,\cdots,\widetilde{\lambda}_n \right),\quad
\widetilde{\lambda}_1\geq \cdots \geq \widetilde{\lambda}_n \geq 0,
\end{equation}
Combining \eqref{logAalpha decompose}, \eqref{HAalpha} and \eqref{logdetHA}, and introducing $P_0$ defined in \eqref{parameter}, we have
\begin{small}
\begin{equation}\label{finY}
\begin{aligned}
&{\rm Tr}(YY^*)-\sum_{\alpha=1}^t
c_{\alpha}\log\det(A_{\alpha})=\frac{n t_0}{\tau}
-\frac{n^2}{N}-n\sum_{\alpha=1}^tc_{\alpha}\log (f_{\alpha}+t_0)
\\
&+\frac{N^{-1}}{2}
\sum_{\alpha=1}^t\frac{c_{\alpha}}{(f_{\alpha}+t_0)^2}
\left(
\overline{z_0-a_{\alpha}}^2{\rm Tr}\big(
\hat{Z}^2
\big)+
(z_0-a_{\alpha})^2{\rm Tr}\big(
\hat{Z}^*
\big)^2-2t_0 U_1U_2\hat{Z}U_2^*U_1^*\hat{Z}^*
\right)    \\
&-N^{-\frac{1}{2}}\sum_{\alpha=1}^t\frac{c_{\alpha}}{f_{\alpha}+t_0}
\left(
\overline{z_0-a_{\alpha}}
{\rm Tr}\big(
\hat{Z}
\big)+
(z_0-a_{\alpha})
{\rm Tr}\big(
\hat{Z}^*
\big)
\right)   
 +\frac{P_1}{2} {\rm Tr}\big(
\widetilde{\Lambda}^2
\big) \\
&-N^{-\frac{1}{2}}\left(
\overline{P_0}{\rm Tr}\big(
U_2^*\widetilde{\Lambda}U_2\hat{Z}
\big)+P_0{\rm Tr}\big(
U_1\widetilde{\Lambda}U_1^*\hat{Z}^*
\big)
\right)+O\Big(\big(
N^{-\frac{1}{2}}+\| \widetilde{\Lambda} \|
\big)^3\Big).
\end{aligned}
\end{equation}
\end{small}
From the above we know the typical size of $\widetilde{\Lambda}$ is $N^{-\frac{1}{2}}$.

Combining \eqref{fTY}, \eqref{Talphaexpansion1}, \eqref{Talphaexpansion5}, \eqref{Tuestimate}-\eqref{diagonal1} and \eqref{finY}, we have
\begin{equation}\label{fTYexpanNon0}
f(T,Y,Q)-f_0\big(\tau;\big\{\frac{\tau c_{\alpha}}{f_{\alpha}+t_0}\mathbb{I}_n\big\},\sqrt{t_0}U_1U_2,0\big)
=F_0+F_1+O(F_2),
\end{equation}
where
\begin{equation*}
\begin{aligned}
F_0&=\frac{N^{-\frac{1}{2}}}{\tau}\left(
\overline{z_0}
{\rm Tr}\big(
\hat{Z}
\big)+
z_0
{\rm Tr}\big(
\hat{Z}^*
\big)
\right) +\frac{n^2}{N}       \\
&-\frac{N^{-1}}{2}
\sum_{\alpha=1}^t\frac{c_{\alpha}}{(f_{\alpha}+t_0)^2}
\left(
\overline{z_0-a_{\alpha}}^2{\rm Tr}\big(
\hat{Z}^2
\big)+
(z_0-a_{\alpha})^2{\rm Tr}\big(
\hat{Z}^*
\big)^2
\right)  ,
\end{aligned}
\end{equation*}
\begin{small}
\begin{equation}\label{f1Non0}
\begin{aligned}
&F_1=\frac{|z_0|^2+t_0}{\tau}{\rm Tr}(\Gamma_1)
-\frac{1}{\tau}
{\rm Tr}\Big(Q_{t+1}
\Big(\big(
z_0\mathbb{I}_{r_{t+1}}-A_{t+1}
\big)^*
\big(
z_0\mathbb{I}_{r_{t+1}}-A_{t+1}
\big)+t_0\mathbb{I}_{r_{t+1}}
\Big)
Q_{t+1}^*\Big)
  \\
&-
\frac{t_0}{\tau}{\rm Tr}(Q_{0}Q_0^*)
-\sum_{\alpha=2}^t\frac{(f_{\alpha}+t_0)^2}{2\tau^2 c_{\alpha}}{\rm Tr}\big(\widetilde{T}_{{\rm d},\alpha}^2  \big)        
 -\frac{(f_{1}+t_0)^2}{2\tau^2 c_{1}}{\rm Tr}
\Big(
\sum_{\alpha=2}^t
\widetilde{T}_{{\rm d},\alpha}
 \Big)^2
\\&
+\frac{N^{-\frac{1}{2}}}{\tau}\sum_{\alpha=2}^t{\rm Tr}\Big(
\big(\overline{a_{\alpha}-a_1}\hat{Z}
+(a_{\alpha}-a_1)\hat{Z}^*\big)
 \widetilde{T}_{{\rm d},\alpha}
\Big)
-\sum_{\alpha=2}^t
\frac{f_{\alpha}+t_0}{\tau}{\rm Tr}\big(
T_{{\rm u},\alpha}T_{{\rm u},\alpha}^*
\big)\\&
-\frac{(f_1+t_0)^2}{\tau^2 c_1}
 {\rm Tr}\Big(
 \sum_{\alpha=2}^t
 \sqrt{\frac{\tau c_{\alpha}}{f_{\alpha}+t_0}}T_{{\rm u},\alpha}
\Big)
 \Big(
\sum_{\alpha=2}^t
 \sqrt{\frac{\tau c_{\alpha}}{f_{\alpha}+t_0}}T_{{\rm u},\alpha}
\Big)^*
 \\&+\frac{1}{\tau}\sum_{1\leq i < j\leq n}
 \Big|
\sum_{\alpha=2}^t
\sqrt{\frac{\tau c_{\alpha}}{f_{\alpha}+t_0}}
(a_{\alpha}-a_1)
t_{i,j}^{(\alpha)}
 \Big|^2+\frac{t_0}{N}P_1{\rm Tr}\big(
 U_1U_2\hat{Z}U_2^*U_1^*\hat{Z}^*
 \big)   \\
 &-\frac{P_1}{2} {\rm Tr}\big(
\widetilde{\Lambda}^2
\big)+N^{-\frac{1}{2}}\Big(
\overline{P_0}{\rm Tr}\big(
U_2^*\widetilde{\Lambda}U_2\hat{Z}
\big)+P_0{\rm Tr}\big(
U_1\widetilde{\Lambda}U_1^*\hat{Z}^*
\big)
\Big),
\end{aligned}
\end{equation}
\end{small}
and
\begin{small}
\begin{equation*}
\begin{aligned}
F_2=\big(
N^{-\frac{1}{2}}+\| \widetilde{\Lambda} \|
\big)^3
+ \Omega_{{\rm error}}^3+
N^{-\frac{1}{2}}\big(
\|\Gamma_1\|+\Omega_{{\rm error}}^2
\big)+\big(
\|\Gamma_1\|+\|G_1\|
\big)\Omega_{{\rm error}}.
\end{aligned}
\end{equation*}
\end{small}
From \eqref{Omegaerror} and \eqref{change scale} we see that the typical size of $F_2$ is $N^{-\frac{3}{2}}$, so $O(F_2)$ is the error term.
 
{\bf Step 2: Taylor expansion of $g(T,Y,Q)$.} For $g(T,Y,Q)$ defined in \eqref{gYUT},  we write down  leading terms  for five relevant  factors while the estimate for  the last determinant   in \eqref{gYUT}  is  non-trivial. Note $R_{\alpha,N}$ is defined in \eqref{ralpha N}.
 \begin{itemize}
 \item[(1)]For the first factor in  \eqref{gYUT},  by \eqref{matrix transformations1} and \eqref{matrix transformations2},
\begin{small}
 \begin{equation}\label{detTQ}
\Big(
\det\big(
\mathbb{I}_n-\sum_{\alpha=1}^tT_{\alpha}T_{\alpha}^*
-Q_0Q_0^*-Q_{t+1}Q_{t+1}^*
\big)
\Big)^{R_0-n}=\Big(\det\big(
-\Gamma_1-G_1-G_1^*
\big)\Big)^{R_0-n},
\end{equation}
\end{small}
  \item[(2)]For the second factor in  \eqref{gYUT}, by \eqref{zi}, \eqref{singularcorre} and \eqref{tildeLambda},
 \begin{equation}\label{detZY}
\det\begin{bmatrix}
\sqrt{\gamma_N}Z & -Y^*   \\
Y  &   \sqrt{\gamma_N}Z^*
\end{bmatrix}=
(|z_0|^{2}+t_0)^n+O\big(
N^{-\frac{1}{2}}+\| \widetilde{\Lambda} \|
\big).
\end{equation} 
  \item[(3)]For the third factor in  \eqref{gYUT} with $\alpha=1,\cdots,t,$, it is easy to see from \eqref{Tdalpha1} and \eqref{Omegaerror} that
\begin{small}
 \begin{equation}\label{tjjalpha}
 \prod_{\alpha=1}^t\prod_{j=1}^n
\big(
t_{j,j}^{(\alpha)}
\big)^{R_{\alpha,N}-j}=
\prod_{\alpha=1}^t\prod_{j=1}^n
\Big(
   \frac{\tau c_{\alpha}}{f_{\alpha}+t_0}
+O\big(\Omega_{\rm error}
\big)\Big)^{R_{\alpha,N}-j}.
\end{equation}
\end{small}
 \item[(4)]For the fourth factor in  \eqref{gYUT}, by \eqref{zi} and \eqref{A alpha}
 \begin{equation}\label{detAalpha}
\big(\det(A_{\alpha})\big)^{R_{\alpha,N}-n}=
\Big((f_{\alpha}+t_0)^n
+O\big(
N^{-\frac{1}{2}}+\| \widetilde{\Lambda} \|
\big)\Big)^{R_{\alpha,N}-n}.
\end{equation}

  \end{itemize}
  
The remaining part of this subsection  is devoted to    the last determinant   in \eqref{gYUT}.

Recalling $\hat{L}_1$ and $\hat{L}_2$ defined in    \eqref {L1hat} and \eqref{L0hat}, since all $a_{\alpha}\neq z_0$ and $t_0>0$, $A_{{\alpha}}$ is  invertible when $N$ is sufficiently large. Rewrite
\begin{equation*}
\det\!\bigg(
\widehat{L}_1+\sqrt{\gamma_N}\widehat{L}_2
\bigg)=\left[\begin{smallmatrix}
P_{1,1} & O\big( \|\widetilde{Q}_{t+1}\| \big)   \\
O\big( \|\widetilde{Q}_{t+1}\| \big) &
\widehat{B}_{t+1}+O\big( \|\widetilde{Q}_{t+1}\|^2 \big)
 \end{smallmatrix}
\right],
\end{equation*}
where
\begin{small}
\begin{equation*}
P_{1,1}=\left[\begin{smallmatrix}
A_1\otimes \mathbb{I}_n &&\\
& \ldots & \\
&& A_t\otimes \mathbb{I}_n
 \end{smallmatrix}
\right]+
\sqrt{\gamma_N}
\left[\begin{smallmatrix}
\mathbb{I}_{2n}\otimes T_1^*\\
\vdots \\
 \mathbb{I}_{2n}\otimes T_t^* 
 \end{smallmatrix}
\right]
\left[\begin{smallmatrix}
\left[\begin{smallmatrix}
a_1\mathbb{I}_n & \\ & \overline{a}_{1}\mathbb{I}_n
\end{smallmatrix}\right]\otimes T_1 &
,\cdots, &
\left[\begin{smallmatrix}
a_{t}\mathbb{I}_n & \\ & \overline{a}_{t}\mathbb{I}_n
\end{smallmatrix}\right]\otimes T_t
\end{smallmatrix}\right],
\end{equation*}
\end{small}
change the order of the above matrix product and  we obtain 
\begin{small}
 \begin{equation*}
\begin{aligned}
&\det(P_{1,1})=\Big(\prod_{\alpha=1}^t \det\big( A_{\alpha} \big) \Big)^n
\det\bigg(
\mathbb{I}_{2n^2}+\sqrt{\gamma_N} 
\sum_{\alpha=1}^t
\left( A_{\alpha}^{-1} \left[\begin{smallmatrix}
a_{\alpha}\mathbb{I}_n & \\ & \overline{a}_{\alpha}\mathbb{I}_n
\end{smallmatrix}\right] \right) \otimes
(T_{\alpha}T_{\alpha}^*)
\bigg),
\end{aligned}
\end{equation*}
\end{small}
From \eqref{singularcorre} and \eqref{tildeLambda}, we have $ Y=\sqrt{t_0}U_1U_2+O(\| \widetilde{\Lambda} \|), $ in view of \eqref{A alpha}, rewrite
\begin{small}
\begin{equation*}\label{A alpha rewrite}
A_{\alpha}=\begin{bmatrix}
Z_{\alpha} & -Y^* \\ Y & Z_{\alpha}^*
\end{bmatrix},\quad
Z_{\alpha}=\sqrt{\gamma_N}
(Z-a_{\alpha}\mathbb{I}_n),
\end{equation*} 
\end{small}
then
\begin{small}
\begin{equation*}\label{tildeA alpha inverse}
\begin{aligned}
A_{\alpha}^{-1}
=\begin{bmatrix}
\Big(
Z_{\alpha}+Y^*\big(Z_{\alpha}^*\big)^{-1}Y
\Big)^{-1}
&
Z_{\alpha}^{-1}Y^*
\big( Z_{\alpha}^*+YZ_{\alpha}^{-1}Y^* \big)^{-1}
\\
-\big(Z_{\alpha}^*\big)^{-1}Y
\Big(
Z_{\alpha}+Y^*\big(Z_{\alpha}^*\big)^{-1}Y
\Big)^{-1}
&
\big( Z_{\alpha}^*+YZ_{\alpha}^{-1}Y^* \big)^{-1}
\end{bmatrix}.
\end{aligned}
\end{equation*}
\end{small}
Simple calculations show that
\begin{small}
\begin{equation*}
\Big(
Z_{\alpha}+Y^*\big(Z_{\alpha}^*\big)^{-1}Y
\Big)^{-1}=\overline{z_0-a_{\alpha}}(f_{\alpha}+t_0)^{-1}\mathbb{I}_n
+O\big( N^{-\frac{1}{2}}+\| \widetilde{\Lambda} \| \big),
\end{equation*}
\begin{equation*}
T_{\alpha}T_{\alpha}^*=\frac{\tau c_{\alpha}}{f_{\alpha}+t_0}\mathbb{I}_n
+O\big( 
\Omega_{\rm error}
 \big).
\end{equation*}
\end{small}
Rewrite $\frac{a_{\alpha}}{z_0-a_{\alpha}}=
\frac{a_{\alpha}\overline{z_0-a_{\alpha}}}{f_{\alpha}}$, we have
\begin{small}
\begin{equation*}
\begin{aligned}
\mathbb{I}_{2n^2}
&+\sqrt{\gamma_N} 
\sum_{\alpha=1}^t
\left( A_{\alpha}^{-1} \left[\begin{smallmatrix}
a_{\alpha}\mathbb{I}_n & \\ & \overline{a}_{\alpha}\mathbb{I}_n
\end{smallmatrix}\right] \right) \otimes
(T_{\alpha}T_{\alpha}^*)=
\\&
\tau
\begin{bmatrix}
(t_0P_1-z_0\overline{P}_0)
\mathbb{I}_{n^2} &  \sqrt{t_0}(\overline{P}_0+\overline{z}_0P_1)
((U_2^*U_1^*)\otimes \mathbb{I}_n)
\\
-\sqrt{t_0}(P_0+z_0P_1)
((U_1U_2)\otimes \mathbb{I}_n)
& (t_0P_1-\overline{z}_0P_0)
\mathbb{I}_{n^2}
\end{bmatrix}+O\big(  N^{-\frac{1}{2}}+\Omega_{\rm error}+\| \widetilde{\Lambda} \|    \big),
\end{aligned}
\end{equation*}
\end{small}
where $P_0$ and $P_1$ are defined in \eqref{parameter2}. Also $|t_0P_1-\overline{z}_0P_0|^2+t_0|P_0+z_0P_1|^2=(|z_0|^2+t_0)(|P_0|^2+t_0P_1^2).$ Combining \eqref{detAalpha} we have
\begin{equation}\label{detsum tildeAalpha}
\det\big(  
P_{1,1}
\big)   
=\tau^{2n^2}(|z_0|^2+t_0)^{n^2}(|P_0|^2+t_0P_1^2)^{n^2}\prod_{\alpha=1}^t (f_{\alpha}+t_0)^{n^2}+
O\big( 
\Omega_{\rm error}
+N^{-\frac{1}{2}}+\| \widetilde{\Lambda} \|
 \big).
\end{equation} 
As $t_0>0$, we have $ \det(P_{1,1})\not=0. $
and from $\widetilde{Q}_{t+1}=[Q_0,Q_{t+1}]$, $
\|\widetilde{Q}_{t+1}\|^2=O\big(
\Omega_{\rm error}
\big).
$ Therefore,
\begin{equation*}
\det\!\bigg(
\widehat{L}_1+\sqrt{\gamma_N}\widehat{L}_2
\bigg)=\det(P_{1,1})\det\Big(
\widehat{B}_{t+1}+O\big(
\|\widetilde{Q}_{t+1}\|^2
\big)
\Big).
\end{equation*}
In view of \eqref{B0hat} and \eqref{Q convenience}, from elementary computations,
\begin{small}
\begin{equation}\label{detP22}
\begin{aligned}
\det\Big(
\widehat{B}_{t+1}+O\big(
\|\widetilde{Q}_{t+1}\|^2
\big)
\Big)= t_0^{nr_0}
 \Big(
\det\big(
&(z_0\mathbb{I}_{r_{t+1}}-A_{t+1})^*(z_0\mathbb{I}_{r_{t+1}}-A_{t+1})
+t_0\mathbb{I}_{r_{t+1}}
\big)
\Big)^n
\\&+O\big( 
\Omega_{\rm error}
+N^{-\frac{1}{2}}+\| \widetilde{\Lambda} \|
 \big),
 \end{aligned}
\end{equation}
\end{small}
Combing \eqref{detsum tildeAalpha}-\eqref{detP22}, we have
\begin{small}
 \begin{equation}\label{detMexpan}
\begin{aligned}
&\det\!\bigg(
\widehat{L}_1+\sqrt{\gamma_N}\widehat{L}_2
\bigg)=\tau^{2n^2}(|z_0|^2+t_0)^{n^2}(|P_0|^2+t_0P_1^2)^{n^2}
t_0^{nr_0}\prod_{\alpha=1}^t
(f_{\alpha}+t_0)^{n^2}
\\&\times \Big(
\det\big(
(z_0\mathbb{I}_{r_{t+1}}-A_{t+1})^*(z_0\mathbb{I}_{r_{t+1}}-A_{t+1})
+t_0\mathbb{I}_{r_{t+1}}
\big)
\Big)^n
 +O\big( 
\Omega_{\rm error}
+N^{-\frac{1}{2}}+\| \widetilde{\Lambda} \|
 \big)  
 ,
\end{aligned}
\end{equation}
\end{small}

Now recalling \eqref{gYUT}, combining \eqref{detTQ}-\eqref{detAalpha} and \eqref{detMexpan} we have
\begin{small}
\begin{equation}\label{gTYQexpansion}
\begin{aligned}
&g(T,Y,Q)=\Big(\det\big(
-\Gamma_1-G_1-G_1^*
\big)\Big)^{R_0-n}
\bigg((|z_0|^{2}+t_0)^{nR_0}
\tau^{n\sum_{\alpha=1}^tR_{\alpha,N}-\frac{n(n+1)}{2}t+2n^2}
\\&
\times 
(|P_0|^2+t_0P_1^2)^{n^2}
t_0^{nr_0}\prod_{\alpha=1}^t \Big(
c_{\alpha}^{nR_{\alpha,N}-\frac{(n+1)n}{2}}
(f_{\alpha}+t_0)^{\frac{(n+1)n}{2}}\Big)
\\&\times
\Big(\det\big(
(z_0\mathbb{I}_{r_{t+1}}-A_{t+1})^*(z_0\mathbb{I}_{r_{t+1}}-A_{t+1})
+t_0\mathbb{I}_{r_{t+1}}
\big)
\Big)^n
 +O\big( 
\Omega_{\rm error}
+N^{-\frac{1}{2}}+\| \widetilde{\Lambda} \|
 \big)\bigg).
\end{aligned}
\end{equation}
\end{small}

{\bf Step 3: Removal of error terms: A standard procedure of Laplace method.} Finally, with those preparations  in {\bf Step 1} and {\bf Step 2}, we will  give a  compete  proof of Theorem \ref{2-complex-correlation}.
 
Recalling \eqref{INdelta}, \eqref{Jacobiandet} and \eqref{fTYexpanNon0}, we rewrite
 \begin{equation}\label{INdeltaexpan}
e^{-NF_0}I_{N,\delta}=J_{1,N}+J_{2,N},
\end{equation}
 where 
\begin{small}
 \begin{equation}\label{J1N}
J_{1,N}=
\frac{\pi^{n^2}}{\big( \prod_{i=1}^{n-1}i! \big)^2}
\Big(
\frac{f_{1}+t_0}{\tau c_{1}}
\Big)^{\frac{n(n-1)}{2}}\int_{\hat{\Omega}_{N,\delta}}
g(T,Y,Q)\big(1+O\big(
\Omega_{\rm error}
\big)\big)e^{NF_1}
\prod_{1\leq i<j\leq n}(\widetilde{\lambda}_{j}-\widetilde{\lambda}_{i})^2
{\rm d}\widetilde{V},
\end{equation}
\end{small}
and
\begin{small}
 \begin{equation}\label{J2N}
J_{2,N}=
\frac{\pi^{n^2}}{\big( \prod_{i=1}^{n-1}i! \big)^2}
\Big(
\frac{f_{1}+t_0}{\tau c_{1}}
\Big)^{\frac{n(n-1)}{2}}\int_{\hat{\Omega}_{N,\delta}}
g(T,Y,Q)\big(1+O\big(
\Omega_{\rm error}
\big)\big)e^{NF_1}
\prod_{1\leq i<j\leq n}(\widetilde{\lambda}_{j}-\widetilde{\lambda}_{i})^2
\big( e^{O(NF_2)}-1 \big){\rm d}\widetilde{V},
\end{equation}
\end{small}
with $Y=U_1\sqrt{t_0\mathbb{I}_n+\widetilde{\Lambda}}U_2,$
\begin{equation}\label{widetildeV}
{\rm d}\widetilde{V}={\rm d}\widetilde{\Lambda}
{\rm d}U_1{\rm d}U_2
{\rm d}Q_0{\rm d}Q_{t+1}{\rm d}\Gamma_1
{\rm d}G_1\prod_{\alpha=2}^t {\rm d}\widetilde{T}_{{\rm d},\alpha}
{\rm d}T_{{\rm u},\alpha}
\end{equation}
and 
\begin{equation}\label{hat Omega N delta}
\hat{\Omega}_{N,\delta}=A_{N,\delta}\cap \big\{
\Gamma_1+G_1+G_1^*\leq 0
\big\}\cap \big\{
\widetilde{\lambda}_n\leq \cdots \leq \widetilde{\lambda}_1
\big\},
\end{equation}
 cf.  \eqref{ANdelta} for definition of $A_{N,\delta}$.
 
Under  the restriction condition  of $\Gamma_1+G_1+G_1^* \leq 0$,  every   principal minor  of order 2 is non-positive definite, so all diagonal entries of $\Gamma_1$  are zero or negative and  
\begin{equation}\label{g2hatU1 estimation}
{\rm Tr}\big(
G_1G_1^*
\big)=\sum_{1\leq i<j\leq n}\big| (G_1)_{i,j} \big|^2
\leq  \sum_{1\leq i<j\leq n} (\Gamma_1)_{i,i} (\Gamma_1)_{j,j}
\leq  
\big(  {\rm Tr}
(\Gamma_1)
\big)^2.
\end{equation}
Noticing the absence of $ G_1$ in $F_1$ given  in \eqref{f1Non0}, for convergence we need to control  $G_1$ by $
\Gamma_1$  in $F_2$,  just as shown in  \eqref{g2hatU1 estimation}. Since  $\| G_1 \|= O(\| \Gamma_1 \|)$ from   \eqref{g2hatU1 estimation}, also from \eqref{regiondelta} and \eqref{ANdelta} for sufficiently large $N$ and small $\delta$, there exists some $C>0$ such that 
  \begin{small}
\begin{equation*}
\begin{aligned}
&\frac{1}{C} NF_2 \leq N^{-\frac{1}{2}}+\sqrt{\delta}+
\sqrt{N\delta} \big(
\| \widetilde{\Lambda} \|+\|\widetilde{T}_{{\rm d}} \|_2
+\| Q_0 \|+\| Q_{t+1} \|+\|T_{{\rm u}} \|_2
\big)            \\
&+N\sqrt{\delta}\big(
\| Q_0 \|^2+\| \widetilde{\Lambda} \|^2
+\|\widetilde{T}_{{\rm d}} \|_2^2+\|T_{{\rm u}} \|_2^2
+\| Q_{t+1} \|^2+\| \Gamma_1 \|
\big).
\end{aligned}
\end{equation*}
\end{small}

 Using the inequality \begin{equation}\label{J2N inequ}
\big|
e^{O(NF_2)}-1
\big|\leq O(N|F_2|)e^{O(N|F_2|)},
\end{equation}
 after change of  variables
\begin{equation}\label{change scale}
\begin{aligned}
&(G_1,\Gamma_1)\rightarrow N^{-1}(G_1,\Gamma_1),\quad
(Q_0,Q_{t+1},\widetilde{\Lambda})\rightarrow
 N^{-\frac{1}{2}}(Q_0,Q_{t+1},\widetilde{\Lambda}),
\\
&\big(
\widetilde{T}_{{\rm d},\alpha},T_{{\rm u},\alpha}
\big)\rightarrow N^{-\frac{1}{2}} 
\big(
\widetilde{T}_{{\rm d},\alpha},T_{{\rm u},\alpha}
\big)
\quad \alpha=2,\cdots,t,
\end{aligned}
\end{equation}
the term $O(NF_2)$ in \eqref{J2N inequ} has an upper bound  by   $N^{-\frac{1}{2}}P(\tilde{V})$ for some polynomial of variables.  
since $F_1$ can control $F_2$  for sufficiently  small $\delta$, by the argument of  Laplace method and the dominant convergence theorem we know that 
$J_{2,N}$ is typically of order  $N^{-\frac{1}{2}}$ compared with $J_{1,N}$, that is,  
 \begin{equation}\label{J2Nestimation}
J_{2,N}=O\big(
N^{-\frac{1}{2}}
\big)J_{1,N}.
\end{equation}
For $J_{1,N}$,   
take  a large  $M_0>0$ such that
\begin{small}
\begin{equation*}
\begin{aligned}
A_{N,\delta}^{\complement}\subseteq  
&   \bigcup_{\alpha=2}^t
\Big\{
{\rm Tr}\big(
T_{{\rm u},\alpha}T_{{\rm u},\alpha}^*
\big)>\frac{\delta}{M_0}
\Big\}
\bigcup
\bigcup_{\alpha=2}^t
\Big\{
{\rm Tr}\big(
\widetilde{T}_{{\rm d},\alpha}\widetilde{T}_{{\rm d},\alpha}^*
\big)>\frac{\delta}{M_0}
\Big\} \bigcup \Big\{{\rm Tr}\big(
\Gamma_{1}\Gamma_{1}^*
\big)>\frac{\delta}{M_0}\Big\}
        \\
&\bigcup \Big\{{\rm Tr}\big(
Q_0Q_0^*
\big)>\frac{\delta}{M_0}\Big\} 
 \bigcup
\Big\{
{\rm Tr}\big(
\widetilde{\Lambda}^2
\big)>\frac{\delta}{M_0}
\Big\}
\bigcup
\Big\{
{\rm Tr}\big(
Q_{t+1}Q_{t+1}^*
\big)>\frac{\delta}{M_0}
\Big\}.
\end{aligned}
\end{equation*}
\end{small}
Here we have used \eqref{g2hatU1 estimation} to drop out the domain   
$ \{{\rm Tr}\big(
G_1G_1^*
\big)>\delta/M_0\}$.   For each piece of  domain,  only  keep  the restricted matrix variable and let  the others free,  it's easy to prove that the corresponding  matrix integral  is exponentially  small, that is,  $O\big( e^{-\delta_1 N} \big)
$
for some $\delta_1>0$.  

So we can extend the integration  region from $\hat{\Omega}_{N,\delta}$ to  $\big\{ \Gamma_1+G_1+G_1^* \leq 0 \big\}\cap \big\{
\widetilde{\lambda}_n\leq \cdots \leq \widetilde{\lambda}_1
\big\}$. From \eqref{ralpha N} and \eqref{calpha sum}, we have $\sum_{\alpha=1}^t R_{\alpha,N}=-r_0-r_{t+1}-R_0$, and by  the change of  variables    
\eqref{change scale} we have
\begin{equation}\label{J1Nform}
\begin{aligned}
J_{1,N}&=N^{-\frac{n^2 t}{2}-n(R_0+r_0+r_{t+1})}
\frac{\pi^{n^2}}{\big( \prod_{i=1}^{n-1}i! \big)^2}
\Big(
\frac{f_1+t_0}{\tau c_1}
\Big)^{\frac{n(n-1)}{2}}
\tau^{2n^2-n(r_0+r_{t+1}+R_0)-\frac{n(n+1)}{2}t }
       \\
&\times 
(|P_0|^2+t_0P_1^2)^{n^2}
t_0^{nr_0}
 \Big(
\det\big(
(z_0\mathbb{I}_{r_{t+1}}-A_{t+1})^*(z_0\mathbb{I}_{r_{t+1}}-A_{t+1})
+t_0\mathbb{I}_{r_{t+1}}
\big)
\Big)^n \\
&\times  (|z_0|^2+t_0)^{n^2}   
\prod_{\alpha=1}^t c_{\alpha}^{nR_{\alpha,N}-\frac{n(n+1)}{2}}
(f_{\alpha}+t_0)^{\frac{n(n+1)}{2}}\Big(
I_0+O\big(
N^{-\frac{1}{2}}
\big)
\Big),
\end{aligned}
\end{equation}
where with $ \widehat{H}=\Gamma_1+G_1+G_1^*$,
\begin{equation}\label{I0}
I_0=\int_{\widehat{H}\leq 0}
\int_{\widetilde{\lambda}_n\leq \cdots \leq \widetilde{\lambda}_1}
\big(
\det(-\widehat{H})
\big)^{R_0-n}
\prod_{1\leq i<j\leq n}(\widetilde{\lambda}_{j}-\widetilde{\lambda}_{i})^2
e^F
{\rm d}\widetilde{V},
\end{equation}
\begin{small}
\begin{equation}\label{I0F}
\begin{aligned}
&F=\frac{|z_0|^2+t_0}{\tau}{\rm Tr}(\Gamma_1)
-\frac{1}{\tau}
{\rm Tr}\Big(Q_{t+1}
\Big(\big(
z_0\mathbb{I}_{r_{t+1}}-A_{t+1}
\big)^*
\big(
z_0\mathbb{I}_{r_{t+1}}-A_{t+1}
\big)+t_0\mathbb{I}_{r_{t+1}}
\Big)
Q_{t+1}^*\Big)
  \\
&-
\frac{t_0}{\tau}{\rm Tr}(Q_{0}Q_0^*)
-\sum_{\alpha=2}^t\frac{(f_{\alpha}+t_0)^2}{2\tau^2 c_{\alpha}}{\rm Tr}\big(\widetilde{T}_{{\rm d},\alpha}^2  \big)        
 -\frac{(f_{1}+t_0)^2}{2\tau^2 c_{1}}{\rm Tr}
\Big(
\sum_{\alpha=2}^t
\widetilde{T}_{{\rm d},\alpha}
 \Big)^2
\\&
+\frac{1}{\tau}\sum_{\alpha=2}^t{\rm Tr}\Big(
\big(\overline{a_{\alpha}-a_1}\hat{Z}
+(a_{\alpha}-a_1)\hat{Z}^*\big)
 \widetilde{T}_{{\rm d},\alpha}
\Big)
-\sum_{\alpha=2}^t
\frac{f_{\alpha}+t_0}{\tau}{\rm Tr}\big(
T_{{\rm u},\alpha}T_{{\rm u},\alpha}^*
\big)\\&
-\frac{(f_1+t_0)^2}{\tau^2 c_1}
 {\rm Tr}\Big(
 \sum_{\alpha=2}^t
 \sqrt{\frac{\tau c_{\alpha}}{f_{\alpha}+t_0}}T_{{\rm u},\alpha}
\Big)
 \Big(
\sum_{\alpha=2}^t
 \sqrt{\frac{\tau c_{\alpha}}{f_{\alpha}+t_0}}T_{{\rm u},\alpha}
\Big)^*
 \\&+\frac{1}{\tau}\sum_{1\leq i < j\leq n}
 \Big|
\sum_{\alpha=2}^t
\sqrt{\frac{\tau c_{\alpha}}{f_{\alpha}+t_0}}
(a_{\alpha}-a_1)
t_{i,j}^{(\alpha)}
 \Big|^2+t_0P_1{\rm Tr}\big(
 U_1U_2\hat{Z}U_2^*U_1^*\hat{Z}^*
 \big)   \\
 &-\frac{P_1}{2} {\rm Tr}\big(
\widetilde{\Lambda}^2
\big)+
\overline{P_0}{\rm Tr}\big(
U_2^*\widetilde{\Lambda}U_2\hat{Z}
\big)+P_0{\rm Tr}\big(
U_1\widetilde{\Lambda}U_1^*\hat{Z}^*
\big),
\end{aligned}
\end{equation}
\end{small}

{\bf Step 4: Matrix integrals and final proof.} To simplify $I_0$ further, we first integrate out $Q_0,Q_{t+1}$ and $T_{{\rm u},2},\cdots,T_{{\rm u},t}$,
\begin{small}
\begin{equation}\label{Qtadd1 integral}
\begin{aligned}
&\int
e^{-\frac{1}{\tau}{\rm Tr}\Big(Q_{t+1}
\Big(\big(
z_0\mathbb{I}_{r_{t+1}}-A_{t+1}
\big)^*
\big(
z_0\mathbb{I}_{r_{t+1}}-A_{t+1}
\big)+t_0\mathbb{I}_{r_{t+1}}
\Big)
Q_{t+1}^*\Big)}{\rm d}Q_{t+1} 
\\&       
=
\Big(\det
\Big(\big(
z_0\mathbb{I}_{r_{t+1}}-A_{t+1}
\big)^*
\big(
z_0\mathbb{I}_{r_{t+1}}-A_{t+1}
\big)+t_0\mathbb{I}_{r_{t+1}}
\Big)
\Big)^{-n}
(\pi\tau)^{n r_{t+1}},
\end{aligned}
\end{equation}
\end{small}
\begin{equation}\label{Qtadd2 integral}
\int
e^{-\frac{t_0}{\tau}{\rm Tr}\big(Q_{0}
Q_{0}^*\big)}{\rm d}Q_{0}       
=
(\pi\tau)^{n r_{0}}t_0^{-nr_0}.
\end{equation}
Calculate the gaussian integrals and we obtain
\begin{small}
\begin{equation}\label{galpha integral}
\begin{aligned}
&\int
\exp\Big\{
-\sum_{\alpha=2}^t
\frac{f_{\alpha}+t_0}{\tau}{\rm Tr}\big(
T_{{\rm u},\alpha}T_{{\rm u},\alpha}^*
\big)
-\frac{(f_1+t_0)^2}{\tau^2 c_1}
 {\rm Tr}\Big(
 \sum_{\alpha=2}^t
 \sqrt{\frac{\tau c_{\alpha}}{f_{\alpha}+t_0}}T_{{\rm u},\alpha}
\Big)
 \Big(
\sum_{\alpha=2}^t
 \sqrt{\frac{\tau c_{\alpha}}{f_{\alpha}+t_0}}T_{{\rm u},\alpha}
\Big)^*
 \\&+\frac{1}{\tau}\sum_{1\leq i < j\leq n}
 \Big|
\sum_{\alpha=2}^t
\sqrt{\frac{\tau c_{\alpha}}{f_{\alpha}+t_0}}
(a_{\alpha}-a_1)
t_{i,j}^{(\alpha)}
 \Big|^2
 \Big\}
 \prod_{\alpha=2}^t {\rm d}T_{{\rm u},\alpha}    
 =
 \big(
\det(A)
 \big)^{-\frac{n(n-1)}{2}}
 \pi^{\frac{n(n-1)}{2}(t-1)},
\end{aligned}
\end{equation}
\end{small}
where for $\alpha,\beta=2,\cdots,t$,
\begin{small}
\begin{equation*}
A_{\alpha,\beta}=\Big(
\frac{(f_1+t_0)^2}{\tau c_1}-(a_{\alpha}-a_1)\overline{a_{\beta}-a_1}
\Big)
\sqrt{\frac{c_{\alpha}c_{\beta}}{(f_{\alpha}+t_0)(f_{\beta}+t_0)}}
+\frac{1}{\tau}(f_{\alpha}+t_0)\delta_{\alpha,\beta}.
\end{equation*}
\end{small}
To compute $\det(A)$, first rewrite A as
\begin{small}
\begin{equation*}
A={\rm diag}\Big(\sqrt{\frac{c_{\alpha}}{f_{\alpha}+t_0}}\Big)_{\alpha=2}^t
\Big( \frac{(f_1+t_0)^2}{\tau c_1}\vec{1}^t\vec{1}+\widehat{A} \Big)
{\rm diag}\Big(\sqrt{\frac{c_{\beta}}{f_{\beta}+t_0}}\Big)_{\beta=2}^t,
\end{equation*}
\end{small}
where
\begin{small}
\begin{equation*}\label{A rewrite 4}
\vec{1}=(1,1,\cdots,1),\quad
\widehat{A}_{\alpha,\beta}=
\frac{(f_{\alpha}+t_0)^2}{\tau c_{\alpha}}\delta_{\alpha,\beta}-
(a_{\alpha}-a_1)\overline{a_{\beta}-a_1},
\quad
\alpha,\beta=2,\cdots,t.
\end{equation*}
\end{small}
Then 
\begin{small}
\begin{equation}\label{A rewrite 5}
\det(A)=\det\big( \widehat{A} \big)
\big( 1+\frac{(f_1+t_0)^2}{\tau c_1}\vec{1}\widehat{A}^{-1}\vec{1}^t \big)
\prod_{\alpha=2}^t\frac{c_{\alpha}}{f_{\alpha}+t_0}
\end{equation}
\end{small}
From basic computations, for $\alpha,\beta=2,\cdots,t$,
\begin{small}
\begin{equation}\label{Ahat inverse}
\begin{aligned}
&\big(\widehat{A}^{-1}\big)_{\alpha,\beta}
=\frac{\tau c_{\alpha}}{(f_{\alpha}+t_0)^2}\delta_{\alpha,\beta}+
\frac{\tau^2 c_{\alpha}c_{\beta}(a_{\alpha}-a_1)\overline{a_{\beta}-a_1}}
{(f_{\alpha}+t_0)^2(f_{\beta}+t_0)^2}
\Big( 1-\sum_{\alpha=2}^t
\frac{\tau c_{\alpha}}{(f_{\alpha}+t_0)^2}|a_{\alpha}-a_1|^2 \Big)^{-1}
,\\&
\det\big( \widehat{A} \big)=
\Big( 1-\sum_{\alpha=2}^t
\frac{\tau c_{\alpha}}{(f_{\alpha}+t_0)^2}|a_{\alpha}-a_1|^2 \Big)
\prod_{\alpha=2}^t
\frac{(f_{\alpha}+t_0)^2}{\tau c_{\alpha}}
.
\end{aligned}
\end{equation}
\end{small}
Recall the definition of $P_0$ and $P_1$ in \eqref{parameter2},
\begin{small}
\begin{equation*}\label{tildeA computation2}
\Big|
\sum_{\alpha=2}^t
\frac{c_{\alpha}}{(f_{\alpha}+t_0)^2}(a_{\alpha}-a_1)
\Big|^2
=\Big|
\sum_{\alpha=2}^t
\frac{c_{\alpha}}{(f_{\alpha}+t_0)^2}\big((a_{\alpha}-z_0)
-(a_{1}-z_0)
\big)
\Big|^2=
|P_0-(a_1-z_0)P_1|^2.
\end{equation*}
\end{small}
Also from $\sum_{\alpha=1}^t
\frac{\tau c_{\alpha}}{f_{\alpha}+t_0}=1,$ 
\begin{small}
\begin{equation*}\label{tildeA computation3}
\begin{aligned}
\sum_{\alpha=2}^t
\frac{c_{\alpha}}
{(f_{\alpha}+t_0)^2}|a_{\alpha}-a_1|^2&=
\sum_{\alpha=1}^t
\frac{c_{\alpha}}{(f_{\alpha}+t_0)^2}\big|(a_{\alpha}-z_0)
-(a_{1}-z_0)
\big|^2     \\
&=\frac{1}{\tau}-t_0P_1-\overline{a_1-z_0}P_0
-(a_1-z_0)\overline{P}_0+|a_1-z_0|^2P_1.
\end{aligned}
\end{equation*}
\end{small}
so from \eqref{Ahat inverse} and basic computations we obtain
\begin{small}
\begin{equation*}\label{tildeA computation4}
\big( 1+\frac{(f_1+t_0)^2}{\tau c_1}\vec{1}\widehat{A}^{-1}\vec{1}^t \big)
\Big( 1-\sum_{\alpha=2}^t
\frac{\tau c_{\alpha}}{(f_{\alpha}+t_0)^2}|a_{\alpha}-a_1|^2 \Big)
=\tau\frac{(f_1+t_0)^2}{c_1}\big( t_0P_1^2+|P_0|^2 \big).
\end{equation*}
\end{small}
Hence,
Combining \eqref{A rewrite 5} and \eqref{Ahat inverse} we obtain \begin{equation}\label{detAfinal}
\det(A)=\tau^{-(t-2)}\frac{f_1+t_0}{c_1}
\big( t_0P_1^2+|P_0|^2 \big)\prod_{\alpha=1}^t (f_{\alpha}+t_0).
\end{equation}

Secondly, to integrate out $G_1$ and $\Gamma_1$, we need the following proposition. 
\begin{proposition}\label{matrixintegral1}
For  a  non-positive definite Hermitian matrix  $H_n=\left[ h_{i,j} \right]_{i,j=1}^n$,  given an integer  $ r_0\ge n$,  then 
\begin{small}
\begin{equation}\label{matrixintegral1equ}
\int_{H_n\leq 0} \left( \det(H_n) \right)^{r_0-n} \prod_{i<j}^n{\rm d}h_{i,j}
=
 \pi^{\frac{n(n-1)}{2}}  ( (r_0-1)! )^{-n} \prod_{k=1}^n   (r_0-k)!  \prod_{j=1}^n  (h_{j,j})^{r_0-1}.
\end{equation}
\end{small}
\end{proposition}
For proof see \cite[Proposition A.2]{LZ23}. From this Proposition, recalling $\widehat{H}=\Gamma_1+G_1+G_1^*$, we have
\begin{small}
\begin{equation}\label{g1 integration}
\int_{\widehat{H}\leq 0}\big(
\det(-\widehat{H})
\big)^{R_0-n}{\rm d}G_1
=(-1)^{n(R_0-n)}\pi^{\frac{n(n-1)}{2}}
\frac{\prod_{i=1}^n(R_0-i)!\big( \det(\Gamma_1) \big)^{R_0-1}}
{\big((R_0-1)!\big)^n},
\end{equation}
\end{small}
from which
\begin{small}
\begin{equation}\label{Gamma integral}
(-1)^{n(R_0-n)}\int_{\Gamma_1\leq 0}
\big( \det(\Gamma_1) \big)^{R_0-1}
e^{\frac{|z_0|^2+t_0}{\tau}{\rm Tr}(\Gamma_1)}{\rm d}\Gamma_1
=\tau^{nR_0}(|z_0|^2+t_0)^{-nR_0}\big((R_0-1)!\big)^n.
\end{equation}
\end{small}
Thirdly, we need to integrate out $\widetilde{T}_{{\rm d},\alpha}$ for $\alpha=2,\cdots,t$. For $i=1,\cdots,n$, let
\begin{equation*}\label{bi}
\vec{b}_i=\big(
b_i^{(2)},b_i^{(3)},\cdots,b_i^{(t)}
\big),\quad
b_i^{(\alpha)}=
\big(
\overline{a_{\alpha}-a_1}\hat{Z}
+(a_{\alpha}-a_1)\hat{Z}^*
 \big)_{i,i},
\end{equation*}
and introduce a square matrix
\begin{equation*}  \label{Sigma}
\Sigma=[\Sigma_{\alpha,\beta}]_{\alpha,\beta=2}^t,  \quad \Sigma_{\alpha,\beta}=\frac{(f_{\alpha}+t_0)^2}{2c_{\alpha}}
\delta_{\alpha,\beta}+\frac{(f_{1}+t_0)^2}{2c_{1}},
\end{equation*} 
Noting
\begin{small}
\begin{equation}\label{Ainverse critical non0}
\det(\Sigma)=2P_1\prod_{\alpha=1}^t \frac{(f_{\alpha}+t_0)^2}{2c_{\alpha}}
,\quad
(\Sigma^{-1})_{\alpha,\beta}=
\frac{2c_{\alpha}\delta_{\alpha,\beta}}{(f_{\alpha}+t_0)^2}
-\frac{2c_{\alpha}c_{\beta}}{(f_{\alpha}+t_0)^2(f_{\beta}+t_0)^2P_1},
\quad
\alpha,\beta=2,\cdots,t,
\end{equation} 
\end{small}
where $P_1$ is given in \eqref{parameter2}. Calculate the gaussian integrals and we obtain
\begin{small}
\begin{equation}\label{Ualpha integral}
\begin{aligned}
&\int \exp
\Big\{
-\frac{(f_{1}+t_0)^2}{2\tau^2 c_{1}}{\rm Tr}
\Big(
 \sum_{\alpha=2}^t
 \widetilde{T}_{{\rm d},\alpha}
 \Big)^2
  -\sum_{\alpha=2}^t\frac{(f_{\alpha}+t_0)^2}{2\tau^2 c_{\alpha}}
  {\rm Tr} \big( \widetilde{T}_{{\rm d},\alpha}^2  \big)
  \\&     
+\frac{1}{\tau}\sum_{\alpha=2}^t{\rm Tr}\Big(
\big(
\overline{a_{\alpha}-a_1}\hat{Z}+
(a_{\alpha}-a_1)\hat{Z}^*
\big)
\widetilde{T}_{{\rm d},\alpha}
\Big)\Big\}
\prod_{\alpha=2}^t{\rm d}\widetilde{T}_{{\rm d},\alpha} \\
&=
\Big(
  2P_1\prod_{\alpha=1}^t \frac{(f_{\alpha}+t_0)^2}{2c_{\alpha}} 
  \Big)
 ^{-\frac{n}{2}}
 \pi^{\frac{t-1}{2}n}
 \tau^{(t-1)n}
e^{\frac{1}{4}\sum_{i=1}^n 
 \vec{b}_i \Sigma^{-1}\vec{b}_i^t}.
\end{aligned}
\end{equation}
\end{small}
Rewrite $$ \overline{a_{\alpha}-a_1}\hat{Z}+
(a_{\alpha}-a_1)\hat{Z}^*=\overline{a_{\alpha}-z_0}\hat{Z}+
(a_{\alpha}-z_0)\hat{Z}^*-\big( \overline{a_{1}-z_0}\hat{Z}+
(a_{1}-z_0)\hat{Z}^* \big), $$
obviously,
\begin{small}
\begin{equation*}
\begin{aligned}
&\sum_{i=1}^n\sum_{\alpha=2}^t
 \frac{c_{\alpha}\big(
 b_i^{(\alpha)}
\big)^2 }
 {(f_{\alpha}+t_0)^2}=
\sum_{\alpha=1}^t
 \frac{c_{\alpha} }
 {(f_{\alpha}+t_0)^2}{\rm Tr}\big(
\overline{a_{\alpha}-z_0}\hat{Z}+
(a_{\alpha}-z_0)\hat{Z}^*\big)^2
\\&
-2{\rm Tr}\Big(\big(
\overline{a_{1}-z_0}\hat{Z}+
(a_{1}-z_0)\hat{Z}^*\big)\big(
\overline{P}_0\hat{Z}+
P_0\hat{Z}^*\big)\Big)+P_1{\rm Tr}\big(
\overline{a_{1}-z_0}\hat{Z}+
(a_{1}-z_0)\hat{Z}^*\big)^2, 
\end{aligned}
\end{equation*}
\end{small}
\begin{small}
\begin{equation*}
\begin{aligned}
\sum_{i=1}^n
\Big(\sum_{\alpha=2}^t
 \frac{c_{\alpha}
 b_i^{(\alpha)} }
 {(f_{\alpha}+t_0)^2}\Big)^2=
{\rm Tr}\Big(
\overline{P}_0\hat{Z}+
P_0\hat{Z}^*-P_1\big(
\overline{a_{1}-z_0}\hat{Z}+
(a_{1}-z_0)\hat{Z}^*\big)
\Big)^2.
 \end{aligned}
\end{equation*}
\end{small}
Use \eqref{Ainverse critical non0}, elementary calculation gives us
\begin{small}
\begin{equation}\label{14biSigmainversebi}
\begin{aligned}
&\frac{1}{4}\sum_{i=1}^n
 \vec{b}_i \Sigma^{-1}\vec{b}_i^t
 =
\frac{1}{2}\sum_{i=1}^n\sum_{\alpha=2}^t
 \frac{c_{\alpha}\big(
 b_i^{(\alpha)}
\big)^2 }
 {(f_{\alpha}+t_0)^2}-\frac{1}{2P_1}\sum_{i=1}^n
 \Big(\sum_{\alpha=2}^t
 \frac{c_{\alpha}
 b_i^{(\alpha)} }
 {(f_{\alpha}+t_0)^2}\Big)^2
\\&=\frac{1}{2}\sum_{\alpha=1}^t
 \frac{c_{\alpha}}
 {(f_{\alpha}+t_0)^2}{\rm Tr}\big(
\overline{a_{\alpha}-z_0}\hat{Z}+
(a_{\alpha}-z_0)\hat{Z}^*\big)^2
-\frac{1}{2P_1}{\rm Tr}\big(
\overline{P}_0\hat{Z}+
P_0\hat{Z}^*
\big)^2,
 \end{aligned}
\end{equation} 
\end{small}

Finally, with \eqref{norm-1} and \eqref{DNn} in mind, by the Stirling's formula we can obtain
\begin{small}
\begin{equation}\label{eNf0 DNn}
\begin{aligned}
&e^{NF_0}D_{N,n}=\frac
{N^{\frac{n^2t}{2}+n(R_0+1+r_0+r_{t+1})}}
{\pi^{n(n+1+r_0+r_{t+1})+\frac{n(n-1)}{2}t}}
(2\pi)^{-\frac{n t}{2}}
\frac{\prod_{\alpha=1}^t c_{\alpha}^{-n R_{\alpha,N}+\frac{n^2}{2}}}
{\prod_{k=1}^n (R_0-k)!}\tau^{-\frac{n(n+1)}{2}}      \\
&\times \prod_{1\leq i<j\leq n}
\big| \hat{z}_i-\hat{z}_j \big|^2
e^{-\frac{1}{\tau}{\rm Tr}\big( \hat{Z}\hat{Z}^* \big)
-\frac{1}{2}\sum_{\alpha=1}^t  
\frac{c_{\alpha}}{(f_{\alpha}+t_0)^2}\Big(
\overline{z_0-a_{\alpha}}^2
{\rm Tr}\big( \hat{Z}^2 \big)+(z_0-a_{\alpha})^2
{\rm Tr}\big( \hat{Z}^* \big)^2
\Big)}
\big( 1+O(N^{-1}) \big).
\end{aligned}
\end{equation}
\end{small}
Now combining \eqref{RNndelta}, \eqref{INdeltaexpan}, \eqref{J2Nestimation}, \eqref{J1Nform}-\eqref{galpha integral}, \eqref{detAfinal}, \eqref{g1 integration}, \eqref{Gamma integral}, \eqref{Ualpha integral}, \eqref{14biSigmainversebi} and \eqref{eNf0 DNn}, we can get
\begin{small}
\begin{equation}\label{RNn final form 1}
\begin{aligned}
&\frac{1}{N^n}
R_N^{(n)}\big(
\tau,A_0;z_0+N^{-\frac{1}{2}}\hat{z}_1
,\cdots,
z_0+N^{-\frac{1}{2}}\hat{z}_n
\big)
=\frac{\prod_{1\leq i<j\leq n}\big|\hat{z}_i-\hat{z}_j\big|^2}
{(\prod_{i=1}^{n-1}i!)^2(2\pi)^{\frac{n}{2}}\pi^n}
\big(t_0P_1^2+|P_0|^2 \big)^{\frac{n(n+1)}{2}}P_1^{-\frac{n}{2}}
\\
&\times 
e^{-\frac{1}{2P_1}{\rm Tr}\big(
\overline{P}_0\hat{Z}+
P_0\hat{Z}^*
\big)^2-t_0P_1{\rm Tr}\big( \hat{Z}\hat{Z}^* \big)}
\int_{\widetilde{\lambda}_n\leq \cdots \leq \widetilde{\lambda}_1}
\prod_{1\leq i<j\leq n}(\widetilde{\lambda}_{j}-\widetilde{\lambda}_{i})^2
\exp\Big\{
t_0P_1{\rm Tr}\big( U_1U_2\hat{Z}U_2^*U_1^*\hat{Z}^* \big)
\\&-\frac{P_1}{2} {\rm Tr}\big(
\widetilde{\Lambda}^2
\big)+
\overline{P_0}{\rm Tr}\big(
U_2^*\widetilde{\Lambda}U_2\hat{Z}
\big)+P_0{\rm Tr}\big(
U_1\widetilde{\Lambda}U_1^*\hat{Z}^*
\big)
\Big\}{\rm d}\widetilde{\Lambda}{\rm d}U_1{\rm d}U_2+O\big(
N^{-\frac{1}{2}}
\big).
\end{aligned}
\end{equation}
\end{small}
After basic computations, the exponential terms in integrand of \eqref{RNn final form 1} now becomes
\begin{small}
\begin{equation*}
\begin{aligned}
&\exp\Big\{
-\frac{P_1}{2}{\rm Tr}\big(
U_2^*\widetilde{\Lambda}U_2-\frac{1}{P_1}
\big(  
\overline{P}_0\hat{Z}+U_2^*U_1^*\hat{Z}^*U_1U_2 
 \big)
\big)^2
\\&+\big(  
t_0P_1+\frac{|P_0|^2}{P_1}
 \big){\rm Tr}\big(  
U_2^*U_1^*\hat{Z}^*U_1U_2 \hat{Z}
 \big)+\frac{\overline{P}_0^2}{2P_1}{\rm Tr}\big(  
\hat{Z}^2
 \big)+\frac{P_0^2}{2P_1}{\rm Tr}\big(  
\hat{Z}^*
 \big)^2
\Big\}.
\end{aligned}
\end{equation*}
\end{small}
Now take the matrix transformations $W=U_1U_2$, $U_2=U_2$, then we have 
${\rm d}U_1{\rm d}U_2={\rm d}W{\rm d}U_2$. And set $S=U_2^*\widetilde{\Lambda}U_2,$ with Jacobian $ {\rm d}S=\frac{\pi^{\frac{n(n-1)}{2}}}{\prod_{i=1}^{n-1}i!}\prod_{1\leq i<j\leq n}(\widetilde{\lambda}_{j}-
\widetilde{\lambda}_{i})^2{\rm d}\widetilde{\Lambda}{\rm d}U_2.  $

Now we can integrate out the Hermitian $S$ and the unitary $W$ matrix variables
\begin{small}
\begin{equation*}
\begin{aligned}
\int_{S=S^*} e^{-\frac{P_1}{2}{\rm Tr}\big(
S-\frac{1}{P_1}
\big(  
\overline{P}_0\hat{Z}+W^*\hat{Z}^*W 
 \big)
\big)^2}{\rm d}S=2^{-\frac{n(n-1)}{2}}\Big(
\frac{2\pi}{P_1}
\Big)^{\frac{n^2}{2}},
\end{aligned}
\end{equation*}
\end{small}
next, from Harish-Chandra-Itzykson-Zuber integration formula  \cite{HC,IZ}  (see also \cite{Me})
\begin{equation}\label{HCIZ}
\int_{\mathcal{U}(n)}\exp\!\left\{
l{\rm Tr}\left( AUBU^* \right)
\right\}{\rm d}U=l^{-\frac{n(n-1)}{2}}
\frac{\det\big( [ e^{l a_i b_j} ]_{i,j=1}^n \big)}
{\prod_{1\leq i<j\leq n}(a_{j}-a_{i}) (b_{j}-b_{i})} \prod_{i=1}^{n-1}i!,
\end{equation}
with  $A={\rm diag}\left( a_1,\cdots,a_n \right)$ and  $B={\rm diag}\left( b_1,\cdots,b_n \right)$, we obtain 
\begin{small}
\begin{equation*}
\begin{aligned}
&\int_{\mathcal{U}(n)}\exp\!\left\{
\big(  
t_0P_1+\frac{|P_0|^2}{P_1}
 \big){\rm Tr}\left( W^*\hat{Z}^*W\hat{Z} \right)
\right\}{\rm d}W=\big(  
t_0P_1+\frac{|P_0|^2}{P_1}
 \big)^{-\frac{n(n-1)}{2}}
\\&\times\prod_{i=1}^{n-1}i!
\prod_{1\leq i<j\leq n}\big|\hat{z}_i-\hat{z}_j\big|^{-2}
\det\Big(
\Big[
e^{\big(  
t_0P_1+\frac{|P_0|^2}{P_1}
 \big)\hat{z}_i\overline{\hat{z}_j}}
\Big]_{i,j=1}^n
\Big).
\end{aligned}
\end{equation*}
\end{small}
Finally, take the rescaling of $\hat{Z}$ by $
\hat{z}_i\rightarrow
\big(  
t_0P_1+\frac{|P_0|^2}{P_1}
 \big)^{-\frac{1}{2}}\hat{z}_i,\quad
 i=1,\cdots,n,
$ we can get to the conclusion.


\hspace*{\fill}

\hspace*{\fill}

 \noindent{\bf Acknowledgements}  
We would like to thank Elton P. Hsu for  his  encouragement and support, Dang-Zheng Liu for useful comments on this paper and Mohammed Osman for useful discussions around the issues.

%
%
%


\begin{thebibliography}{10}
\addtolength{\itemsep}{-1. em} 
\setlength{\itemsep}{-4pt} 

\bibitem{BC16}
Bordenave, C.,   Capitaine, M.  Outlier eigenvalues for deformed I.I.D. random matrices.
Comm. Pure Appl. Math. 69 (2016), no. 11, 2131-2194.

\bibitem{BPD}
Bordenave, C., Caputo, P., Chafa{\"{i}}, D. Spectrum of Markov generators on sparse random graphs.
 Comm. Pure Appl. Math. 67, 4 (2014), 621-669.

\bibitem{BC12}  Bordenave, C., Chafa\"{i}, D.   Around the circular law.  Probab. Surv. 9 (2012),  1-89.

\bibitem{CP16} 
Capitaine M, P\'{e}ch\'{e} S. Fluctuations at the edges of the spectrum of the full rank deformed GUE. Probab. Theory Relat. Fields  165 (2016), 117-161.

         \bibitem{CES}
Cipolloni,  G.,   Erd\H{o}s, L.,  Schr\H{o}der, D.
Edge universality for non-Hermitian random matrices. Probab. Theory Relat. Fields 179 (2021), 1-28.

\bibitem{Gi}
Ginibre, J. 
  Statistical ensembles of complex,  quaternion and real matrices.
 J.  Math.  Phys. 6(1965), no.3, 440-449. 

\bibitem{Gr}
Grela, J.  Diffusion method in random matrix theory.  J. Phys. A: Math. Theor. {49} (2016), 015201 (18 pages).

\bibitem{HC}
Harish-Chandra, Differential operators on a semisimple Lie algebra.  Amer. J. Math. 79 (1957),  87-120.


\bibitem{IZ}
 Itzykson,  C.,  Zuber,  J.B.  The planar approximation II.  J. Math. Phys. 21 (1980), 411–423.

\bibitem{LZ22} Liu, D.-Z., Zhang, L. Phase transition of eigenvalues in deformed Ginibre ensembles. Preprint arXiv: 2204.13171v2.

\bibitem{LZ23} Liu, D.-Z., Zhang, L. Critical edge statistics for deformed GinUEs. Preprint arXiv: 2311.13227v1.

\bibitem{LZ24} Liu, D.-Z., Zhang, L. Repeated erfc statistics for deformed GinUEs. Preprint arXiv: 2402.14362.

\bibitem{MO} Maltsev, A., Osman, M. Bulk universality for complex non-hermitian matrices with independent and identically distributed entries. 
Preprint arXiv: 2310.11429v4.

\bibitem{Me} 
 Mehta, M.L. 
{\em Random Matrices}. 
\newblock Elsevier, Amsterdam, 3rd edition, 2004.

\bibitem{TV15}
Tao,  T., Vu, V.
\newblock Random matrices: universality of local spectral statistics of
  non-hermitian matrices.
  Ann. Probab. 43 (2015), no.2, 782--874.

\bibitem{TV10}
Tao,  T., Vu, V., Krishnapur, M. 
\newblock Random matrices: Universality of ESDS and the circular law.
\newblock { Ann.  Probab.} 38 (2010), no.5, 2023-2065.




\end{thebibliography}
\end{document}